\documentclass{amsart}[11pt]
\usepackage{amssymb,amscd,amsxtra,verbatim}

\newlength{\cellsz}
\newcounter{cellsize}
\newcommand{\setcellsize}[1]{%
  \setcounter{cellsize}{#1}%
    \setlength{\cellsz}{\value{cellsize}\unitlength}}%
\setcellsize{13}%
\newcommand\cellify[1]{\def\thearg{#1}\def\nothing{}%
\ifx\thearg\nothing
\vrule width0pt height\cellsz depth0pt\else
\hbox to 0pt{{\begin{picture}(\value{cellsize},\value{cellsize})
  \put(0,0){\line(1,0){\value{cellsize}}}
  \put(0,0){\line(0,1){\value{cellsize}}}
  \put(\value{cellsize},0){\line(0,1){\value{cellsize}}}
  \put(0,\value{cellsize}){\line(1,0){\value{cellsize}}} \end{picture} \hss}}\fi%
\vbox to \cellsz{ \vss \hbox to \cellsz{\hss$#1$\hss} \vss}}
\newcommand\tableau[1]{\vcenter{\vbox{\let\\\cr
\baselineskip -16000pt \lineskiplimit 16000pt \lineskip 0pt
\ialign{&\cellify{##}\cr#1\crcr}}}}
\newcommand\tabl[1]{\vtop{\let\\\cr
\baselineskip -16000pt \lineskiplimit 16000pt \lineskip 0pt
\ialign{&\cellify{##}\cr#1\crcr}}}%

\newcommand{\af}{{\mathrm{af}}}
\newcommand{\al}{\alpha}
\newcommand{\A}{\mathbb{A}}

\newcommand{\C}{\mathbb{C}}

\newcommand{\geh}{\mathfrak{g}}
\newcommand{\Gr}{\mathrm{Gr}}
\newcommand{\hh}{\mathfrak{h}}
\newcommand{\id}{\mathrm{id}}
\newcommand{\inner}[2]{\langle #1\,,\,#2\rangle}
\newcommand{\ip}[1]{\langle #1 \rangle}
\newcommand{\K}{{\mathcal{K}}}
\newcommand{\Inv}{\mathrm{Inv}}
\newcommand{\la}{\lambda}

\newcommand{\OO}{{\mathcal{O}}}
\newcommand{\om}{\omega}
\newcommand{\Pa}{{\mathcal{P}}}
\newcommand{\pt}{\mathrm{pt}}
\newcommand{\Q}{{\mathcal{Q}}}
\newcommand{\R}{{\mathbb R}}
\newcommand{\tQ}{\tilde{Q}}
\newcommand{\re}{\mathrm{re}}

\newcommand{\sreg}{\mathrm{sreg}}
\newcommand{\ssreg}{\mathrm{ssreg}}
\newcommand{\ui}{{\underline{i}}}
\newcommand{\We}{\widetilde{W}}
\newcommand{\Z}{\mathbb{Z}}

\newtheorem{prop}{Proposition}
\newtheorem{cor}[prop]{Corollary}
\newtheorem{lem}[prop]{Lemma}
\newtheorem{thm}[prop]{Theorem}

\theoremstyle{remark}
\newtheorem{remark}{Remark}
\newtheorem{ex}{Example}

\numberwithin{remark}{section} \numberwithin{ex}{section}
\numberwithin{prop}{section}

\begin{document}
\author{Thomas Lam and Mark Shimozono}
\thanks{T. L. was supported in part by NSF
DMS-0600677.} %
\thanks{M. S. was supported in part by NSF DMS-0401012.}
 \email{tfylam@math.harvard.edu} \email{mshimo@vt.edu}
\title{Quantum cohomology of $G/P$ and homology of affine Grassmannian}
\begin{abstract}
Let $G$ be a simple and simply-connected complex algebraic group, $P
\subset G$ a parabolic subgroup. We prove an unpublished result of
D.~Peterson which states that the quantum cohomology $QH^*(G/P)$ of
a flag variety is, up to localization, a quotient of the homology
$H_*(\Gr_G)$ of the affine Grassmannian $\Gr_G$ of $G$. As a
consequence, all three-point genus zero Gromov-Witten invariants of
$G/P$ are identified with homology Schubert structure constants of
$H_*(\Gr_G)$, establishing the equivalence of the quantum and
homology affine Schubert calculi.

For the case $G = B$, we use the Mihalcea's equivariant quantum
Chevalley formula for $QH^*(G/B)$, together with relationships
between the quantum Bruhat graph of Brenti, Fomin and Postnikov and
the Bruhat order on the affine Weyl group. As byproducts we obtain
formulae for affine Schubert homology classes in terms of quantum
Schubert polynomials. We give some applications in quantum
cohomology.

Our main results extend to the torus-equivariant setting.
\end{abstract}

\maketitle
\section{Introduction}
Let $G$ be a simple and simply-connected complex algebraic group, $P
\subset G$ a parabolic subgroup and $T$ a maximal torus.  This paper
studies the relationship between the quantum cohomology $QH^*(G/P)$
of the flag variety of $G$ and the homology $H_*(\Gr_G)$ of the
affine Grassmannian $\Gr_G$ of $G$.  We show that $QH^*(G/P)$ is a
quotient of $H_*(\Gr_G)$ after localization and describe the map
explicitly on the level of Schubert classes.
As a consequence, all three-point genus zero Gromov-Witten
invariants of $G/P$ are identified with homology Schubert structure
constants of $H_*(\Gr_G)$, establishing the equivalence of the
quantum and homology affine Schubert calculi.  This is an
unpublished result stated by Dale Peterson in 1997 \cite{Pet}.
Peterson's statement and our proof extends to the $T$-equivariant
setting, though
Peterson was not using the definition of equivariant quantum
cohomology in use today.

Quantum Schubert calculus has been studied heavily and we will not
attempt to survey the literature.  The combinatorial study of the
{\it equivariant} quantum cohomology rings $QH^T(G/P)$ is however
more recent (see \cite{Mih}).  Schubert calculus on the affine
Grassmannian was first studied by Kostant and Kumar~\cite{KK} as a
special case of their general study of the topology of Kac-Moody
flag varieties.  That the nilHecke ring of Kostant and Kumar could
be used to study both the homology and cohomology of the affine
Grassmannian was first realized by Peterson, who should be
considered the father of affine Schubert calculus.  Peterson's work
on affine Schubert calculus is related to his theory of geometric
models for $QH^*(G/P)$, most of which has remained unpublished for a
decade; see however~\cite{Kos} and \cite{Rie} for statements of some
of Peterson's results.  Recently, interest in affine Schubert
calculus was rekindled from a different direction: Shimozono
conjectured and later Lam \cite{Lam} proved that the $k$-Schur
functions of Lapointe, Lascoux and Morse~\cite{LLM}, arising in the
study of Macdonald polynomials, represented homology Schubert
classes of the affine Grassmannian when $G = SL(n)$.

The observation that $QH^*(G/P)$ and $H_*(\Gr_G)$ are related, is
already apparent in the literature.  Ginzburg \cite{Gin} described
the cohomology $H^*(\Gr_G)$ as the enveloping algebra of the Lie
algebra of a unipotent group. The same unipotent group occurs in
Kostant's \cite{Kos} description of $QH^*(G/B)$ as a ring of
rational functions.  More recently, Bezrukavnikov, Finkelberg and
Mirkovic~\cite{BFM} described the equivariant $K$-homology of
$\Gr_G$ and discovered a relation with the Toda lattice.  Earlier
the relation of the Toda lattice with $QH^*(G/B)$ had been
established by Kim~\cite{Kim}. One can already deduce
from~\cite{BFM} and \cite{Kim} that some localizations of
$H_*(\Gr_G)$ and $QH^*(G/B)$ are isomorphic\footnote{Finkelberg 
(private communication) has calculated these localizations in the context of
\cite{BFM}.}. However, such a
statement is insufficient for the enumerative applications to
Schubert calculus.  On the other hand,
even knowing the coincidence of Gromov-Witten invariants with affine
homology Schubert structure constants, the fact that the
identification arises from a ring homomorphism is still unexpected;
for example, the theorems of \cite{BKT,Woo} which compare structure
constants in quantum and ordinary cohomology, are not of this form.
However we note that Lapointe and Morse \cite{LM:QH} defined a ring
homomorphism from the linear span of $k$-Schur functions to the
quantum cohomology of the Grassmannian, which via \cite{Lam} may be
interpreted as sending Schubert classes in the homology of the
affine Grassmannian of $SL_{k+1}$ to quantum Schubert classes.

The paper is naturally separated into the two cases $P = B$ and $P
\neq B$.  For $P = B$, our proof is purely algebraic and
combinatorial, and does not appeal to geometry as in (what we
believe is) Peterson's original intended argument, though much of
the combinatorics we develop may well have been known to Peterson.
At the core of the our argument is the relationship between the {\it
quantum Bruhat graph}, first studied by Brenti, Fomin and
Postnikov~\cite{BFP} and the Bruhat order on the {\it superregular}
elements of the affine Weyl group, which we study here. Roughly
speaking, an element $x$ of the affine Weyl group $W_\af$ is
superregular if it has a large translation component.  As a
byproduct, we show that the {\it tilted Bruhat orders} in~\cite{BFP}
are all (dual to) induced suborders of the affine Bruhat order.

The algebraic part of our proof relies on known properties of the
ring $QH^*(G/B)$, in particular the fact that it is associative and
commutative.  Apart from these general properties, we need only one
more formula for $QH^*(G/B)$: the equivariant quantum Chevalley
formula originally stated by Peterson~\cite{Pet}, and recently
proved by Mihalcea~\cite{Mih}.  On the side of $H_*(\Gr_G)$, our
computations rely on a homomorphism $j: H_T(\Gr_G) \to Z_{\A_\af}(S)
\subset \A_\af$, where $\A_\af$ is the affine nil Hecke ring of
Kostant and Kumar~\cite{KK} and $Z_{\A_\af}(S)$ (called the {\it
Peterson subalgebra} in~\cite{Lam}) is the centralizer of $S =
H^T(\pt)$. The map $j$ is again due to Peterson. Proofs of its main
properties can be found in~\cite{Lam}.

Our results allow us to give formulae for the affine Schubert
classes as elements of the Peterson subalgebra.  These formulae
involve generating functions over paths in the affine Bruhat order,
or equivalently in the quantum Bruhat graph.  In particular, our
formulae are related to the quantum Schubert polynomials of
\cite{FGP,Mar}.  Each quantum Schubert polynomial gives a formula
for infinitely many affine Schubert classes.

For the case $P \neq B$ we study the Coxeter combinatorics of the
affinization of the Weyl group of the Levi factor of $P$.  We use
this combinatorics to compare the quantum equivariant Chevalley
formulae for $QH^T(G/B)$ and $QH^T(G/P)$, using the comparison
formula of Woodward~\cite{Woo} to refine the Chevalley formula of
\cite{FW,Mih}. Some of the intermediate results we use are stated by
Peterson in~\cite{Pet}.

We use the affine homology Chevalley formula given in~\cite{LS} to
deduce a formula in $QH^*(G/P)$ for multiplication by the quantum
Schubert class $\sigma_P^{r_\theta}$ labeled by the reflection
$r_\theta$ in the highest root. We show that in the case of the
Grassmannian, the ring homomorphism of Lapointe and Morse
\cite{LM:QH} differs from Peterson's map by the {\it strange
duality} of $QH^*(G/P)$ due to Chaput, Manivel and
Perrin~\cite{CMP}.

In the current work we use the maximal torus $T$ in $G$; yet the
affine Grassmannian affords the additional $\C^*$-action given by
loop rotation. In future work we intend to study the Schubert
calculus of the affine Grassmannian with respect to this extra
$\C^*$-equivariance and to pursue $K$-theoretic analogues of
Peterson's theory.

Both the quantum cohomology $QH^*(G/B)$ and homology $H_*(\Gr_G)$
possess additional structures which would be interesting to compare:
for example, $QH^*(G/B)$ has mirror-symmetric constructions and
$H_*(\Gr_G)$ is a Hopf algebra with an action of the nilHecke ring.
The naturality of our main theorem with respect to Schubert classes
suggests that the appearance of the Toda Lattice in \cite{BFM,Kim}
is somehow related to Schubert calculus.

\section{The equivariant quantum cohomology ring $QH^T(G/B)$}
\subsection{Notations}
Let $G$ be a simple and simply-connected complex algebraic group, $B
\subset G$ a Borel subgroup and $T \subset B$ a maximal torus. Let
$\{\al_i\mid i\in I\} \subset \hh^*$ be a basis of simple roots and
$\{\al_i^\vee\mid i\in I\}\in \hh$ a basis of simple coroots, where
$\hh$ is a Cartan subalgebra of the Lie algebra of $G$. Denote by
$Q=\bigoplus_{i\in I} \Z\al_i \subset \hh^*$ and $Q^\vee =
\bigoplus_{i\in I} \Z \al_i^\vee$ the root and coroot lattices. Let
$P=\bigoplus_{i\in I} \Z\om_i\subset \hh^*$ and
$P^\vee=\bigoplus_{i\in I} \Z\om_i^\vee\subset \hh$ be the weight
and coweight lattices, where $\{\om_i\mid i\in I\}$ and
$\{\om_i^\vee\mid i\in I\}$ are the fundamental weights and
coweights, which are the dual bases to $\{\al_i^\vee\mid i\in I\}$
and $\{\al_i\mid i\in I\}$ with respect to the natural pairing
$\inner{\cdot}{\cdot}: \hh \times \hh^* \to \C$.

Let $W$ denote the Weyl group; it is generated by the simple
reflections $\{r_i\mid i\in I\}$. Let $\ell: W \to \Z$ denote the
length function of $W$. We let $w < v$ denote a relation in the
Bruhat order of $W$ and write $w \lessdot v$ if $w < v$ and $\ell(w)
= \ell(v) - 1$. $W$ acts on $\hh^*$ and $\hh$ by
\begin{align*}
  r_i \mu &= \mu - \inner{\al_i^\vee}{\mu} \al_i \qquad \text{for
  $\mu\in \hh^*$} \\
  r_i \la &= \la - \inner{\la}{\al_i} \al_i^\vee \qquad\text{for
  $\la\in \hh$.}
\end{align*}
These actions stabilize the lattices $Q\subset P\subset \hh^*$ and
$Q^\vee\subset P^\vee\subset \hh$ respectively. The pairing
$\inner{\cdot}{\cdot}$ is $W$-invariant: for all $w\in W$, $\mu\in
\hh^*$, and $\la\in \hh$, we have
\begin{align*}
  \inner{w\cdot \la}{w\cdot\mu} =\inner{\la}{\mu}.
\end{align*}
Let $R=W\cdot\{\al_i\mid i\in I\} \subset \hh^*$ be the root system
of $G$. Then $R=R^+\sqcup -R^+$ where $R^+=R\cap\bigoplus_{i\in I}
\Z_{\ge0}\al_i$ is the set of positive roots. For each $\alpha\in R$
there is a $u\in W$ and $i\in I$ such that $\alpha=u\al_i$. Define
the associated coroot $\al^\vee\in Q^\vee$ of $\al$ by
$\al^\vee=u\al_i^\vee$ and the associated reflection of $\alpha$ by
$r_\alpha = u r_i u^{-1}\in W$; they are independent of the choice
of $u$ and $i$.

\subsection{Quantum equivariant Chevalley formula}
Let us denote by $S = H^T(\pt)$ the symmetric algebra of the weight
lattice $P$. Let $\Z[q] = \Z[q_i \mid i \in I]$ be a polynomial ring
for the sequence of indeterminates $q_i$. For $\lambda = \sum_{i\in
I} a_i \, \alpha_i^\vee\in Q^\vee$ with $a_i\in\Z_{\ge0}$ we set
$q_\la = \prod_{i \in I}q_i^{a_i}\in\Z[q]$. The (small) equivariant
quantum cohomology $QH^T(G/B)$ is isomorphic to $H^T(G/B) \otimes_\Z
\Z[q]$ as a free $\Z[q]$-module, with basis the equivariant quantum
Schubert classes $\{\sigma^{w} \in QH^T(G/B) \mid w \in W\}$. It is
equipped with a quantum multiplication denoted $*: QH^T(G/B) \times
QH^T(G/B) \to QH^T(G/B)$.  This multiplication is associative and
commutative.

When we set $q_i= 1$ in $QH^T(G/B)$ we obtain the usual equivariant
cohomology $H^T(G/B)$.  When we apply the evaluation $\phi_0: S \to
\Z$ at 0 to $QH^T(G/B)$ we obtain the usual quantum cohomology
$QH^*(G/B)$. We refer the reader to \cite{Mih} for more details.  As
shown in~\cite{Mih}, the {\it quantum equivariant Chevalley formula}
completely determines the multiplication in $QH^T(G/B)$.  It was
first stated by Peterson~\cite{Pet} and proved by
Mihalcea~\cite{Mih}.  Define the element $\rho = \sum_{i \in
I}\omega_i = \frac{1}{2} \sum_{\alpha \in R^+} \alpha \in P$.

\begin{thm}[Quantum equivariant Chevalley formula]
\label{thm:monk} Let $i \in I$ and $w \in W$.  Then we have in
$QH^T(G/B)$
$$
\sigma^{r_i} \,*\, \sigma^{w} = (\omega_i - w\cdot \omega_i)
\sigma^w + \sum_{\alpha} \ip{\alpha^\vee, \omega_i}\sigma^{w
r_\alpha}+ \sum_\alpha
\ip{\alpha^\vee,\omega_i}q_{\alpha^\vee}\sigma^{wr_\alpha}
$$
where the first summation is over $\alpha \in R^+$ such that
$wr_\alpha \gtrdot w$ and the second summation is over $\alpha \in
R^+$ such that $\ell(wr_\alpha) = \ell(w) + 1 - \ip{\alpha^\vee,
2\rho}$.
\end{thm}

Our notation here differs slightly from Mihalcea's: the indexing of Schubert bases has
been changed via $w \mapsto w_0w$, and we have made a different choice of positive
roots for $T$.  However, our indexing agrees with the ones in~\cite{FGP, FW, KK}.

Theorem~\ref{thm:monk} can be extended by linearity to give a
formula for the multiplication by the quantum equivariant class
$[\lambda] \in QH^T(G/B)$ of a line bundle with weight $\lambda$.
Theorem~\ref{thm:monk} then corresponds to the case $\lambda =
\omega_i$.  Let us denote by $c_{u,v}^{w,\lambda} \in S$ the
equivariant Gromov-Witten invariants given by
$$
\sigma^v * \sigma^u = \sum_{w \in W} c_{u,v}^{w,\lambda} q_\lambda\, \sigma^w
$$
in $QH^T(G/B)$.  The non-equivariant Gromov-Witten invariants have
an explicit enumerative interpretation which we will not describe
here.

\subsection{Quantum Bruhat graph}
The {\it quantum Bruhat graph} $D(W)$ of \cite{BFP} is the directed
graph with vertices given by the elements of the Weyl group $W$,
with a directed edge from $w$ to $v=wr_\alpha$ for $w\in W$ and
$\alpha\in R^+$ if either $\ell(v) = \ell(w) + 1$ or $\ell(v) =
\ell(w)+1 - \inner{\alpha^\vee}{2\rho}$.

Given $u \in W$, the {\it tilted Bruhat order} $D_u(W)$ of
\cite{BFP} is the graded partial order on $W$ with the relation $w
\prec_u v$ if and only if there is a shortest path in $D(W)$ from
$u$ to $v$ which passes through $w$.  Note that $D_\id(W)$ is the
usual Bruhat order. We refer the reader to \cite[Section 6]{BFP} for
further details.

\section{Affine weyl group}
Let $W_\af = W \ltimes Q^\vee$ denote the affine Weyl group
corresponding to $W$.  For $\lambda \in Q^\vee$, its image in
$W_\af$ is denoted $t_\lambda$. We have $t_{w\cdot \la} = w t_\la
w^{-1}$ for all $w\in W$ and $\la\in Q^\vee$. As a Coxeter group
$W_\af$ is generated by simple reflections $\{r_i \mid i \in
I_\af\}$ where $I_\af = I \sqcup \{0\}$. We denote by $Q_\af =
\oplus_{i \in I_\af} \Z \al_i\subset \hh_\af^*$ and $Q_\af^\vee =
\oplus_{i \in I_\af} \Z \al_i^\vee\subset \hh_\af$ the affine root
and coroot lattices, where $\hh_\af$ is the Cartan subalgebra of the
affine Lie algebra $\geh_\af$ associated to the Lie algebra $\geh$
of $G$. Restriction yields a natural map $Q_\af \to Q$ denoted
$\beta\mapsto \bar\beta$; its kernel is spanned by the null root
$\delta = \sum_{i \in I_\af} a_i \alpha_i = \alpha_0 + \theta$ where
$\theta\in R$ is the highest root. In particular we have $Q_\af
\cong Q \oplus \Z\delta$. Abusing notation we sometimes write
$\alpha$ both for an element of $Q_\af$ and its image $\bar{\alpha}$
in $Q$.

The affine root system $R_\af$ is comprised of the nonzero elements
of the form $\beta = \alpha + n \delta$ where $\alpha \in R \cup
\{0\}$ and $n \in \Z$. The set of positive affine roots $R_\af^+$
consists of the elements $\alpha + n \delta\in R_\af$ such that
either $n>0$ or both $\alpha \in R^+$ and $n = 0$.

Let $R_\af^\re= W_\af\cdot \{\al_i\mid i\in I_\af\}$ be the set of
real roots of $\geh_\af$; it consists of the elements $\beta\in
R_\af$ such that $\bar{\beta}\ne0$. The associated coroot of
$\beta\in R^\af_\re$ is defined by $\beta^\vee = u \al_i^\vee\in
Q_\af^\vee$ where $u\in W_\af$ and $i\in I_\af$ are such that
$\beta=u\al_i$; $\beta^\vee$ is independent of the choice of $u$ and
$i$. The associated reflection is defined by $r_\beta = u r_i
u^{-1}\in W_\af$.

The level zero action of $W_\af$ on $P\oplus \Z\delta$ is given by
\begin{align} \label{E:WaffQaff}
wt_\lambda \cdot (\mu + n \delta) = w\cdot \mu +
(n-\inner{\la}{\mu})\delta
\end{align}
for $w\in W$, $\la\in Q^\vee$, $\mu\in P$ and $n\in \Z$. This action
stabilizes $Q_\af$. For $\beta=\alpha+n\delta\in R_\af^\re$, with
respect to $W_\af = W \ltimes Q^\vee$ one has
\begin{align} \label{E:affinereflection}
  r_{\beta} = r_\alpha t_{n\alpha^\vee}
\end{align}
and, in particular,
\begin{align*}
  r_0 = r_\theta t_{-\theta^\vee}.
\end{align*}
For $x\in W_\af$, define
\begin{align*}
  \Inv(x) = \{\beta\in R_\af^+\mid x\cdot \beta \in -R_\af^+\};
\end{align*}
the elements of $\Inv(x)$ are called \textit{inversions} of $x$. It
is well-known that $\ell(x) = |\Inv(x)|$ for all $x\in W_\af$. The
following standard formula gives the length of $x = wt_\lambda$. It
is obtained by calculating the number of values of $n \in \Z$, for
each fixed $\alpha \in R^+$ such that $\alpha + n \delta\in
\Inv(x)$.

\begin{lem}
\label{lem:length} Let $x = w t_\lambda \in W_\af$.  Then
$$
\ell(x) = \sum_{\alpha \in R^+} |\chi(w \cdot \alpha < 0)+
\ip{\lambda,\alpha} |
$$
where $\chi(P) = 1$ if $P$ is true and $\chi(P) = 0$ otherwise.
\end{lem}

We call $\la \in \hh$ {\it antidominant} if $\inner{\la}{\al_i} \le
0$ for each $i \in I$, and denote by $\tQ$ the set of antidominant
elements of $Q^\vee$. The following lemma is an immediate
consequence of Lemma~\ref{lem:length}.

\begin{lem}
Let $\la \in Q^\vee$ and $w\in W$ such that $w \cdot \la \in \tQ$.
Then $\ell(t_\la) = \inner{w\cdot \la}{-2\rho}$.
\end{lem}

Let $W_\af^-$ denote the set of \textit{Grassmannian} elements in
$W_\af$, which by definition are those that are of minimum length in
their coset in $W_\af/W$. They are characterized below.

\begin{lem}
\label{lem:grass} Let $w\in W$ and $\la\in Q^\vee$. Then $w t_\la
\in W_\af^-$ if and only if $\la\in\tQ$ and $w$ is
\textit{$\la$-minimal}, that is, for every $i \in I$, if
$\inner{\la}{\al_i}=0$ then $w\al_i>0$ (equivalently, $w$ is of
minimum length in its coset in $W/W_\la$ where $W_\la$ is the
stabilizer of $\la$). In this case $\ell(wt_\la) = \ell(t_\la) -
\ell(w)$.
\end{lem}
\begin{proof}
We have $wt_\lambda \in W_\af^-$ if and only if $wt_\lambda \cdot
\alpha_i > 0$ for each $i \in I$.  By \eqref{E:WaffQaff} this holds
if and only if for each $i \in I$ either $\inner{\la}{\alpha_i} < 0$
or $\inner{\la}{\alpha_i} = 0$ and $w \cdot \alpha_i \in R^+$.  This
is exactly the stated condition.  To calculate $\ell(wt_\lambda)$ in
this case one observes that for each $\alpha \in R^+$ we have
$\chi(w\cdot \alpha < 0) + \inner{\la}{\alpha} \leq 0$, so by Lemma
\ref{lem:length}, $\ell(t_\la) - \ell(wt_\la)$ is equal to the
number of inversions of $w$.
\end{proof}

We say that $\la \in Q^\vee$ is regular if the stabilizer $W_\la$ is
trivial.

\begin{lem} \label{lem:wla}
For $\la\in \tQ$ regular,
$$
\ell(ut_{w \cdot \la}) = \ell(t_{\la}) - \ell(uw) + \ell(w).
$$
\end{lem}
\begin{proof}
We have $ut_{w\cdot \la} = uwt_\la w^{-1}$.  By
Lemma~\ref{lem:grass}, $uwt_\la \in W_\af^-$ and $\ell(uwt_\la) =
\ell(t_\la) - \ell(uw)$.  But $\ell(uwt_\la w^{-1}) = \ell(uwt_\la)
+ \ell(w^{-1})$ and $\ell(w^{-1}) = \ell(w)$ so the claim follows.
\end{proof}

The following result can be found in \cite[Lemma 4.3]{BFP} and
\cite[Lemma 3.2]{Mar}.

\begin{lem}
\label{lem:BFP} For any positive root $\alpha \in R^+$, we have
$\ell(r_\alpha) \leq \ip{\alpha^\vee, 2\rho} -1$.  In the case of a
simple laced root system, equality always holds.
\end{lem}

\section{The superregular affine Bruhat order}
We call an element $\lambda \in Q^\vee$ \textit{superregular} if
$|\inner{\lambda}{\alpha}| \gg 0$ for every $\alpha \in
R^+$.\footnote{For our purposes $|\inner{\lambda}{\alpha}| > 2|W| +
2$ is sufficient.} In particular, superregular elements are regular.
We say that $x = wt_\lambda\in W_\af$ is superregular if $\lambda$
is. We fix once and for all a set of superregular elements
$W_\af^\sreg \subset W_\af$.

In the rest of the paper we will say a property, or result holds for
``sufficiently superregular'' elements $W_\af^\ssreg \subset W_\af$
if there is a positive constant $k \in \Z$ such that the property,
or result holds for all $x \in W_\af^\sreg$ satisfying
\begin{quote}
if $y \in W_\af$ satisfies $y < x$ and $\ell(x) - \ell(y) < k$ then
$y \in W_\af^\sreg$.
\end{quote}
We will in general not specify the constant $k$ explicitly but the
computation of $k$ will in every case be trivial.  The notation
$W_\af^\ssreg$ will thus depend on context.

We say that $x = wt_{v\lambda} \in W_\af$ is in the {\it
$v$-chamber} if $\lambda$ is regular antidominant.  We will say that
$x$ and $x'$ are in the same chamber if they are both in the
$v$-chamber for some $v \in W$.


\begin{prop}\label{prop:superregular}
Let $\lambda \in \tQ$ be antidominant and superregular and let $x =
wt_{v\lambda}$. Then $y = xr_{v\alpha + n \delta} \lessdot x$ if and
only if one of the following conditions holds:
\begin{enumerate}
\item
$\ell(wv) = \ell(wvr_\alpha) - 1$ and $n = \inner{\lambda}{\alpha}$,
giving $y = wr_{v\alpha}t_{v\lambda}$.
\item
$\ell(wv) = \ell(wvr_\alpha) + \inner{\alpha^\vee}{2\rho} - 1$ and
$n = \inner{\lambda}{\alpha} + 1$ giving $y =
wr_{v\alpha}t_{v(\lambda+\alpha^\vee)}$.
\item
$\ell(v) = \ell(vr_\alpha) + 1$ and $n = 0$, giving $y =
wr_{v\alpha}t_{vr_\alpha \cdot \lambda}$

\item
$\ell(v) = \ell(vr_\alpha) - \inner{\alpha^\vee}{2\rho} + 1$ and $n
= - 1$, giving $y = wr_{v\alpha}t_{vr_\alpha(\lambda+\alpha^\vee)}$.
\end{enumerate}

\end{prop}

\begin{proof}
Suppose $y = xr_{v\alpha + n \delta} \lessdot x$.  For $n \in \Z$,
define $f(n):= \ell(t_{v( \lambda + n \alpha^\vee)})$.  By
Lemma~\ref{lem:length}, we have $f(n) = \sum_{\beta \in R^+} |
\ip{\lambda + n \alpha^\vee, v^{-1} \cdot \beta}|$ which is a convex
function of $n$.  By superregularity of $\lambda$, we have
\begin{equation}
\label{E:fn} f(n) = f(0) - n\ip{\alpha^\vee, 2\rho}
\end{equation}
for sufficiently small values of $n$.  Also we have
$$
f(-\ip{\lambda,\alpha}) = \ell(t_{vr_\alpha \cdot \lambda}) = f(0)
$$
and thus by superregularity
$$
f(-\ip{\lambda,\alpha} - n) = f(0) - n\ip{\alpha^\vee, 2\rho}
$$
for sufficiently small values of $n$.  By convexity we conclude that
if $n$ is not close to either $0$ or $-\ip{\lambda,\alpha}$ then
$f(n)$ is not close to $f(0)$.  Now write
$$
y = wt_{v\lambda}r_{v\alpha} t_{nv\alpha^\vee} = wr_{v\alpha}
t_{v(\lambda + (n - \inner{\lambda}{\alpha})\alpha^\vee)} .
$$
Since $|\ell(wr_{v\alpha}  t_{v(\lambda + (n -
\inner{\lambda}{\alpha})\alpha^\vee)}) - \ell(t_{v(\lambda + (n -
\inner{\lambda}{\alpha})\alpha^\vee)})| \leq |W|$ by superregularity
and convexity we may thus assume that either (a) $\lambda + (n -
\inner{\lambda}{\alpha})\alpha^\vee$ is antidominant, or (b)
$\lambda - n\alpha^\vee$ is antidominant.  In case (a), using
Lemma~\ref{lem:wla}
\begin{align*}
\ell(y) &= \ell(wvr_\alpha  t_{\lambda + (n -
\inner{\lambda}{\alpha})\alpha^\vee}v^{-1}) \\
&= -\ell(wvr_\alpha)+\ell(v^{-1}) + \ell(t_\lambda) +
(n - \inner{\lambda}{\alpha})\inner{\alpha^\vee}{2\rho} \\
&= \ell(x) +\ell(wv) - \ell(wvr_\alpha) + (n -
\inner{\lambda}{\alpha})\inner{\alpha^\vee}{2\rho}.
\end{align*}
Using Lemma~\ref{lem:BFP}, we deduce that $n =
\inner{\lambda}{\alpha}$ or $n = \inner{\lambda}{\alpha} + 1$ giving
cases (1) and (2) of the Lemma.  Similarly, in case (b), we obtain
cases (3) and (4) of the Lemma.
\end{proof}

Fix a sufficiently superregular antidominant element $\la \in \tQ$.
Let $G_\la$ denote the graph obtained from the restriction of the
Hasse diagram of the Bruhat order on $W_\af$ to the superregular
elements $x \in
W_\af^\sreg$ such that $x \le t_{w\lambda}$ for some $w \in W$.  
We will further direct the edges of $G_\lambda$ downwards (in the
direction of smaller length), so that the $|W|$ vertices $x =
t_{v\lambda}$ are the source vertices.  By
Lemma~\ref{prop:superregular} the edges of $G_\lambda$ either stay
within the same chamber (cases (1) and (2)) or go between different
chambers (cases (3) and (4)).  We call the first kind of edge (or
cover) {\it near} and denote such a cover by $y \lessdot_n x$ and
call the second kind {\it far}, denoting them by $y \lessdot_f x$.
By definition the graph obtained from $G_\lambda$ by keeping only
the near edges is a union of the connected components $G_\lambda^v$
which contain $t_{v\lambda}$, for $v \in W$.

The following combinatorial result makes explicit the relationship
between the quantum Bruhat graph and the superregular affine Bruhat
order.

\begin{cor}\label{cor:path}
Suppose $\lambda \in \tQ$ is sufficiently superregular.    Each edge
$wt_{v\lambda} \to wr_{v\alpha}t_{v\lambda}$ (or $wt_{v\lambda} \to
wr_{v\alpha}t_{v(\lambda+\alpha^\vee)}$) in $G_\lambda^v$ is
canonically associated to the edge $wv \to wvr_\alpha$ in $D(W)$.
Thus each sufficiently short path $\Pa$ in $D(W)$ from $v$ to $w$
induces a unique path $\Q$ in $G_\lambda^w$, which goes from
$t_{v\lambda}$ to $wv^{-1}t_{v\mu}$ where $\mu$ equals $\la$ plus
the sum of $\alpha^\vee$ over all edges in $\Q$ which are of type
(2) (as in Proposition~\ref{prop:superregular}).
\end{cor}
\begin{proof}
The result follows from comparing the definition of the quantum
Bruhat graph with cases (1) and (2) of
Proposition~\ref{prop:superregular}.
\end{proof}

We use the phrase ``sufficiently short'' in Corollary \ref{cor:path}
since a very long path $\Pa$ in $D(W)$ will give rise to a path $\Q$
which leaves $G_\lambda$, that is, uses non-superregular elements.

\begin{remark} \label{R:covers}
\begin{enumerate}
\item
In all cases of Proposition \ref{prop:superregular}, the positive
affine root for the reflection $r_{v\alpha+n\delta}$ is given by
$-v\alpha-n\delta$.
\item
Every superregular element has a unique factorization $ w t_\la
v^{-1}$ where $v,w\in W$ and $\la$ is antidominant superregular. In
passing to a Bruhat cocover of $wt_\la v^{-1}$, $\la$ either stays
the same or is replaced by $\la+\alpha^\vee$; in the ``near" case
$w$ is replaced by $wr_\alpha$ with associated quantum Bruhat edge
$w\to wr_\alpha$, while in the ``far" case $v$ is replaced by
$vr_\alpha$, with associated quantum Bruhat edge
$vr_\alpha\rightarrow v$.
\end{enumerate}
\end{remark}

Given a (sufficiently short) path $\Pa \in D(W)$ beginning at $w \in
W$, we denote by $x_\Pa \in W_\af$ the endpoint of the path in
$G_\lambda^w$ associated to $\Pa$ via Corollary~\ref{cor:path}.  The
following Lemma is a translation of \cite[Lemma 1]{Pos} into our
language.

\begin{lem}\label{lem:Pos}
Suppose $\Pa$ and $\Pa'$ are two paths in $D(W)$ from $w$ to $v$ of
shortest length.  Then $x_\Pa = x_{\Pa'}$.
\end{lem}

\begin{thm}\label{thm:tilted}
Each tilted Bruhat order $D_u(W)$ is dual to an induced suborder of
affine Bruhat order.
\end{thm}
\begin{proof}
Let $x(u,w) \in W_\af$ be the vertex of $G_\lambda^u$ (with
$\lambda$ sufficiently superregular) satisfying $x(u,w) = x_\Pa$ for
a shortest path $\Pa$ from $u$ to $w$ in $D(W)$. By
Lemma~\ref{lem:Pos}, $x(u,w)$ does not depend on the choice of
$\Pa$.  By Proposition \ref{prop:superregular}, the partial order
$D_u(W)$ is canonically isomorphic via the map $w \mapsto x(u,w)$ to
the dual of the affine Bruhat order restricted to elements $\{x(u,w)
\in W_\af \mid w \in W\}$.
\end{proof}

\section{Affine Bruhat operators}
For $X\subset W_\af$ let $S[X]=\bigoplus_{x\in X} S x$ be the free
left $S$-module with basis $X$. For each $\mu \in P$ and
$x=wt_{v\la}\in W_\af^\ssreg$, the \textit{near equivariant affine
Bruhat operator} is the left $S$-module homomorphism
$B^\mu:S[W_\af^\ssreg]\to S[W_\af^\sreg]$ defined by
$$
B^\mu(x) = (\mu - wv\cdot\mu)\,x + \sum_{\alpha \in R^+}
\sum_{xr_{v\alpha+n\delta} \lessdot_n x} \inner{\alpha^\vee}{\mu}\,
xr_{v\alpha+n\delta}
$$

Fix a superregular antidominant element $\lambda \in \tQ$.  We call
an element $\sigma$ of $QH^T(G/B)$ {\it $\lambda$-small} if all
powers $q_\mu$ which occur in $\sigma$ satisfy the property that
$\mu + \lambda$ is superregular antidominant. For each $w \in W$,
define the left $S$-module homomorphism $\Theta_w^\lambda$ from the
$\lambda$-small elements of $QH^T(G/B)$ to $S[G_\lambda]$ by
$$
\Theta^\la_w(q_\mu\,\sigma^v) = vw^{-1}t_{w(\lambda +\mu)} = v t_\mu
(t_\lambda w^{-1}).
$$

The equivariant affine Bruhat operator is related to the equivariant
quantum Chevalley formula via the following result.

\begin{prop} \label{prop:bruhat} Let $\lambda \in \tQ$ be superregular,
$\mu \in P$, $\sigma \in QH^T(G/B)$ be $\lambda$-small, and $w \in W$.  Then
$$
\Theta_w^\lambda(\sigma * [\mu]) = B^\mu (\Theta_w^\lambda(\sigma))
$$
whenever $\Theta_w^\lambda(\sigma)$ is in the domain of $B^\mu$.
\end{prop}
\begin{proof}
By linearity it suffices to prove the statement for $\sigma = q_\mu
\sigma_v$.  We have
\begin{eqnarray*}
&&\Theta_w^\lambda(q_\mu\,\sigma_v *[\mu]) \\
&=&\Theta_w^\lambda(q_\mu((\mu - v\cdot \mu) \sigma^v +
\sum_{\alpha} \ip{\alpha^\vee,\mu}\,\sigma^{v r_\alpha}+ \sum_\alpha
\ip{\alpha^\vee,\mu}\,q_{\alpha^\vee}\sigma^{vr_\alpha})) \\
&=& (\mu-v\cdot\mu)vw^{-1}t_{w(\lambda+\mu)}  \\
&+& \sum_{\alpha} \ip{\alpha^\vee,\mu}vr_\alpha w^{-1}
t_{w(\lambda+\mu)} + \sum_\alpha
\ip{\alpha^\vee,\mu}vr_\alpha w^{-1} t_{w(\lambda+\mu+\alpha^\vee)} \\
&=& B^\mu(vw^{-1}t_{w(\lambda+\mu)}).
\end{eqnarray*}
We have used Theorem~\ref{thm:monk},
Proposition~\ref{prop:superregular}, together with the calculation
$vr_\alpha w^{-1} = vw^{-1}r_{w\alpha}$. The summations in the
equations are as in Theorem~\ref{thm:monk}.
\end{proof}

\begin{prop}\label{prop:commute}
Let $\mu,\nu \in P$.  Then the operators $B^\mu$ and $B^\nu$ commute
as operators on $S[W^\ssreg_\af]$ (whenever they are defined).
\end{prop}
\begin{proof}
Any element $x =  w t_{v\lambda} \in W^\sreg_\af$ is in the image of
$\Theta_v^\lambda$.  The result follows immediately from
Proposition~\ref{prop:bruhat}, since by the commutativity of
$QH^T(G/B)$ one has $\sigma \cdot [\mu] \cdot [\nu] = \sigma \cdot
[\nu] \cdot [\mu]$.
\end{proof}

Let $x = wt_{v\lambda}$.  The \textit{far equivariant affine Bruhat
operator} is the left $S$-module homomorphism
$C^\mu:S[W_\af^\ssreg]\to S[W_\af^\sreg]$ defined by
$$ C^\mu(x) =(\mu- v\cdot\mu)\,x + \sum_{\alpha \in R^+} \sum_{xr_{v\alpha+n\delta} \lessdot_f
x} \inner{\alpha^\vee}{\mu}\, x r_{v\alpha+n\delta}.
$$
The operators $C^\mu$ and $B^\mu$ are related by the following
formula when acting on the special element $\sum_{w\in W}
t_{w\lambda}$.

\begin{lem}\label{lem:BC}
Let $\lambda \in \tQ$ be a sufficiently superregular antidominant coweight and
$\mu^1,\mu^2,\ldots,\mu^k \in P$ be a sequence of integral weights.
Then
\begin{equation} \label{eq:comm}
C^\mu \left(B^{\mu^k} \cdots B^{\mu^2} B^{\mu^1} \cdot \sum_{w\in W}
t_{w\lambda}\right) = B^{\mu^k} \cdots B^{\mu^2} B^{\mu^1} \cdot
\left(B^\mu \cdot \sum_{w\in W} t_{w\lambda}\right).
\end{equation}
\end{lem}
\begin{proof}
A term of $B^{\mu^k} \cdots B^{\mu^2} B^{\mu^1} \cdot t_{w\lambda}$
is indexed by a multipath (a path allowed to stay at a vertex for multiple steps)
$$
\Pa = \{w = w^{(0)} \to w^{(1)} \to w^{(2)} \to \cdots \to w^{(k)}\}
$$
in $D(W)$, where for each $i \in [1,k]$, we have (i) $w^{(i)} =
w^{(i-1)}$ or (ii) $w^{(i)} = w^{(i-1)}r_{\alpha^{(i)}}$.  Each such
path $\Pa$ contributes a term $a_\Pa \, x_\Pa$, where $a_\Pa =
\prod_i a_i$ with $a_i =  \mu^{(i)}- w^{(i)}\cdot\mu^{(i)}$ in case
(i) and $a_i = \inner{(\alpha^{(i)})^\vee}{\mu}$ in case (ii).  The
left hand side of (\ref{eq:comm}) can thus be given as the sum over
pairs $(\Pa,\Q)$ where $\Pa$ is a multipath from $w$ to $v$ in
$D(W)$ of length $k$, and $\Q$ is a multipath from $u$ to $w$ of
length 1.  If $x_\Pa = vw^{-1}t_{w\mu}$ then $(\Pa,\Q)$ contributes
$a_{\Pa,\Q}x_{\Pa,\Q}$ where $x_{\Pa,\Q} = vu^{-1}t_{u\mu'}$ with
$\mu' = \mu$ or $\mu' = \mu + \alpha^\vee$ for some $\alpha \in
R^+$.  The coefficient $a_{\Pa,\Q}$ is equal to $a_\Pa \, a_\Q$
where $a_\Q = \mu - w\cdot\mu$ if $u = w$ and $a_\Q =
\inner{\alpha^\vee}{\mu}$ if $u = wr_\alpha$.

To obtain (\ref{eq:comm}) we send the pair $(\Pa,\Q)$ to the
multipath
$$
\Pa'=\{u \to w = w^{(0)} \to w^{(1)} \to w^{(2)} \to \cdots \to
w^{(k)}\}
$$
and we observe that $x_{\Pa'} = x_{\Pa,\Q}$ and $a_{\Pa'} =
a_{\Pa,\Q}$, where $\Pa'$ is weighted according to the sequence
$\mu,\mu^{(1)},\ldots,\mu^{(k)}$. Note that in the case that
$u=wr_\alpha$, the first step of $\Pa'$ corresponds to a cover $x
r_{(wr_\alpha)\alpha+n\delta} \lessdot x$ where $x=t_{wr_\alpha
\la}$.
\end{proof}

\section{Homology of affine Grassmannian}
\subsection{Affine nilHecke ring}
Let $\A_\af$ denote the {\it affine nilHecke ring} of Kostant and
Kumar.  Our conventions here differ slightly from those in \cite{KK}
but agree with those in \cite{Lam}, and we refer to the latter for a
discussion of the differences. We use the action of $W_\af$ on $P$
induced by the action \eqref{E:WaffQaff}, under which translation
elements act trivially, or equivalently, $r_0$ acts by $r_\theta$.
$\A_\af$ is the ring with a $1$ given by generators $\{A_i \mid i
\in I_\af\} \cup \{\lambda \mid \lambda \in P\}$ and the relations
\begin{align*}
A_i \,\lambda &= (r_i \cdot \lambda)\, A_i +
\ip{\lambda,\alpha_i^\vee}\cdot1 && \text{for $\lambda \in P$,} \\
A_i\, A_i &= 0, \\
\underbrace{A_iA_jA_i\dotsm}_m &= \underbrace{A_jA_iA_j\dotsm}_m &&
\text{if $\underbrace{r_ir_jr_i\dotsm}_m =
\underbrace{r_jr_ir_j\dotsm}_m$,}
\end{align*}
where the ``scalars'' $\lambda \in P$ commute with other scalars.
Let $w \in W_\af$ and let $w = r_{i_1} \cdots r_{i_l}$ be a reduced
decomposition of $w$.  Then $A_w := A_{i_1} \cdots A_{i_l}$ is a
well defined element of $\A_\af$, where $A_\id = 1$. $\A_\af$ is a
free left $S$-module (and a free right $S$-module) with basis $\{A_w
\mid w \in W_\af\}$.  Note that we have
$$
A_xA_y = \begin{cases} A_{xy} &\mbox{if $\ell(x) + \ell(y) =
\ell(xy)$,} \\ 0 & \mbox{otherwise.} \end{cases}$$ We have the
following commutation relation which can be established by
induction; see \cite{KK}.

\begin{lem}For $x \in W_\af$ and $\la \in P$,
\label{lem:commute}
$$A_x \lambda = (x \cdot \lambda) A_x + \sum_{\substack{\beta\in R_\af^{\re+} \\
xr_{\beta} \lessdot x}} \ip{\beta^\vee,\lambda} A_{xr_\beta}.$$
\end{lem}


\subsection{Equivariant homology of affine Grassmannian}
The affine Grassmannian $\Gr_G$ associated to $G$ is the ind-scheme
$G(\K)/G(\OO)$ where $\K = \C((t))$ denotes the ring of formal
Laurent series and $\OO = \C[[t]]$ is the ring of formal power
series.  The space $\Gr_G$ is weakly homotopy equivalent to the
space $\Omega K$ of based loops into the maximal compact subgroup $K
\subset G$ and thus the homology $H_*(\Gr_G)$ and equivariant
homology $H_T(\Gr_G)$ inherits a ring structure via Pontryagin
multiplication.

The ring $H_T(\Gr_G)$ is a free $S= H_T(\pt)$-module with basis
given by the $T$-equivariant Schubert classes $\{\xi_x \mid x \in
W_\af^-\}$. The affine nilHecke ring $\A_\af$ acts on $H_T(\Gr_G)$
by
\begin{align} \label{E:NilHeckeOnHom}
 A_y\cdot \xi_z = \begin{cases} \xi_{yz} &
\text{if $\ell(yz) = \ell(y) + \ell(z)$ and $yz \in W_\af^-$,} \\
0 & \text{otherwise,}
\end{cases}
\end{align}
and $S \subset \A_\af$ acts via the usual $S$-module structure of
$H_T(\Gr_G)$.

We now describe Peterson's model for $H_T(\Gr_G)$ \cite{Pet}. We
refer the reader to~\cite{Lam} for more details.

Let $Z_{\A_\af}(S) \subset \A_\af$ denote the centralizer of $S$ in
$\A_\af$, called the {\it Peterson subalgebra} in \cite{Lam}.  Let
$J \subset \A_\af$ denote the left ideal
$$
J = \sum_{w \in W \setminus \{\id\}} \A_\af A_w.
$$
The following two theorems are due to Peterson \cite{Pet}. We refer
the reader to \cite[Lemma 3.3 and Theorem 4.4]{Lam} for a proof of
Theorem~\ref{thm:pet}.

\begin{thm}
\label{thm:pet} There is an $S$-algebra isomorphism $j: H_T(\Gr_G)
\to
Z_{\A_\af}(S)$ such that %
\begin{align*}
  j(\xi_x) = A_x \mod J \ \ \ \text{and} \ \ \
  j(\xi) \cdot \xi' = \xi \,\xi'
\end{align*}
for $\xi,\xi' \in H_T(\Gr_G)$.
 The element $j(\xi_x)$ is determined by the properties: (1)
$j(\xi_x)\in Z_\A(S)$ and (2) $j(\xi_x) = A_x \mod J$.
\end{thm}
Define $j^y_x \in S$ by
$$
j(\xi_x) = \sum_y j^y_x A_y.
$$
The elements $j^y_x \in S$ are polynomials of degree $\ell(y)-
\ell(x)$ in the simple roots $\{\alpha_i \mid i \in I\}$.

\begin{thm}
\label{thm:j} For $x,z\in W_\af^-$ we have
$$\xi_x\, \xi_z = \sum_y j^y_x\, \xi_{yz}$$
where the summation is over $y \in W_\af$ such that $yz \in W_\af^-$
and $\ell(yz) =\ell(y) + \ell(z)$.
\end{thm}
\begin{proof}
By Theorem~\ref{thm:pet} we have
\begin{align} \label{E:jaction}
  j(\xi_x)\cdot \xi_z = \xi_x\,\xi_z
\end{align}
where the action is as in \eqref{E:NilHeckeOnHom}.  The statement
then follows from the observation that in a length-additive product
$yz \notin W_\af^-$ if $z \notin W_\af^-$.
\end{proof}

\section{Generating elements of the Peterson subalgebra}
We now describe a method for producing elements of the Peterson
subalgebra.  Define the left $S$-module isomorphism $\Upsilon:
S[W_\af] \to \A_\af$ by
\begin{align*}
 \Upsilon(\sum_{x \in W_\af} a_x \, x) =
\sum_{x\in W_\af} a_x \, A_x
\end{align*}
for $a_x\in S$. Let $x = wt_{v\lambda}$.  For $\mu \in P$, the
\textit{twisted equivariant affine Bruhat operators} are the left
$S$-module homomorphisms $\tilde{B}^\mu, \tilde{C}^\mu:
S[W^\ssreg_\af] \to S[W^\sreg_\af]$ defined by
$$ \tilde{B}^\mu(x) =(v^{-1}\mu-w\mu)x + \sum_{\alpha \in R^+} \sum_{xr_{v\alpha+n\delta} \lessdot_n
x} \inner{v\alpha^\vee}{\mu} \,xr_{v\alpha+n\delta}
$$
and
$$ \tilde{C}^\mu(x) = (v^{-1}\mu-\mu)x - \sum_{\alpha \in R^+} \sum_{xr_{v\alpha+n\delta} \lessdot_f
x} \inner{v\alpha^\vee}{\mu} \, x r_{v\alpha+n\delta}.
$$

\begin{lem}\label{lem:center}
Let $f \in S[W^\ssreg_\af]$.  Then $\Upsilon(f) \in Z_{\A_\af}(S)$
if and only if for each $\mu \in P$ we have
$$
\tilde{B}^\mu(f) = \tilde{C}^\mu(f).
$$
\end{lem}
\begin{proof}
The statement follows immediately from Lemma~\ref{lem:commute},
Proposition~\ref{prop:superregular}, and Remark \ref{R:covers} (1).
(If we literally apply Lemma~\ref{lem:commute}, the terms
$(v^{-1}\mu-w\mu)x$ in $ \tilde{B}^\mu(x)$ and $(v^{-1}\mu-\mu)x$ in
$\tilde{C}^\mu(x)$ would need to be negated; since $x$ does not
occur elsewhere in the formula, the stated claim is still true.)
\end{proof}

\begin{thm}\label{thm:central}
Let $\lambda$ be a sufficiently superregular antidominant coweight
and $\mu^1,\mu^2,\ldots,\mu^k \in P$ be a sequence of integral
weights. Then the element
$$
\Upsilon(B^{\mu^k} \cdots B^{\mu^2} B^{\mu^1} \cdot \sum_{w\in W}
t_{w\lambda})
$$
lies in the Peterson subalgebra $Z_{\A_\af}(S)$.
\end{thm}
\begin{proof}
By Lemma~\ref{lem:center}, it suffices to check that
$$
f = B^{\mu^k} \cdots B^{\mu^2} B^{\mu^1} \cdot \sum_{w\in W}
t_{w\lambda}
$$
satisfies $\tilde{B}^\mu(f) = \tilde{C}^\mu(f)$. The coefficient of
$x = wt_{v\lambda}$ in $\tilde{B}^\mu(f)$ (resp. $\tilde{C}^\mu(f)$)
is equal to the coefficient of $x$ in $B^{v^{-1} \cdot \mu}(f)$
(resp. $C^{v^{-1}\mu}(f)$).  Thus it suffices to show that for all
$\mu \in P$ we have $B^\mu(f) = C^\mu(f)$.  But using
Proposition~\ref{prop:commute} this is exactly the statement of
Lemma~\ref{lem:BC}.
\end{proof}

\section{Formulae for affine Schubert classes}
For $w \in W$, let us say that a polynomial
$${\mathfrak S}_w = \sum_{i_1,i_2,\ldots,i_k} a_{i_1,\ldots,i_k}
q_{\lambda(i_1,\ldots,i_k)} \otimes \omega_{i_1}\omega_{i_2} \cdots
\omega_{i_k} \in S[q]\otimes_\Z \Z[\omega_i \mid i \in I], $$ where
$a_{i_1,\ldots,i_k} \in S$ and $\lambda(i_1,\ldots,i_k) \in
\oplus_{i \in I} \Z_{\geq 0} \al_i^\vee$ is an {\it equivariant
quantum Schubert polynomial} if its image in $QH^T(G/B)$ (obtained
by replacing $\omega_i$ with $[\omega_i]$) equals  the quantum
Schubert class $\sigma^w$. There are many choices for such a
polynomial.

Let us write $b(\lambda;\mu^1,\mu^2,\ldots,\mu^k) \in Z_{\A_\af}(S)$
for the element described by Theorem~\ref{thm:central}.  The
following formula writes affine Schubert classes in terms of quantum
Schubert classes.

\begin{thm}\label{thm:schub}
Let ${\mathfrak S}_w\in S[q]\otimes_\Z \Z[\omega_i \mid i \in I] $
as above be an equivariant quantum Schubert polynomial representing
the class $\sigma^w \in QH^T(G/B)$, and let $\lambda \in \tQ$ be
sufficiently superregular.  Then
\begin{equation}\label{eq:schub}
j(\xi_{wt_{\lambda}}) = \sum_{i_1,i_2,\ldots,i_k} a_{i_1,\ldots,i_k}
b(\lambda +\lambda(i_1,\ldots,i_k);
\omega_{i_1},\omega_{i_2},\ldots,\omega_{i_k}).
\end{equation}
\end{thm}
\begin{proof}
Let $a$ denote the expression on the right hand side of
\eqref{eq:schub}. By Theorem~\ref{thm:central}, $a \in
Z_{\A_\af}(S)$. By Theorem~\ref{thm:pet}, it suffices to show that
$a$ contains a unique Grassmannian term $A_{wt_\lambda}$ with
coefficient 1. Let $\ui=(i_1,\dotsc,i_k)$. We have
\begin{align*}
  a &= \Upsilon (\sum_\ui a_\ui  B^{\omega_{i_k}}\dotsm
  B^{\omega_{i_1}} \sum_{w\in W} t_{w(\la+\la(\ui))} ) \\
  &= \Upsilon(\sum_\ui a_\ui B^{\omega_{i_k}}\dotsm
  B^{\omega_{i_1}} (t_{\la+\la(\ui)} + \sum_{w\in W\setminus \{\id\}}
  t_{w(\la+\la(\ui))}))
\end{align*}
By Lemma~\ref{lem:grass} and the fact that $\lambda$ is sufficiently
superregular, it is clear that Grassmannian terms cannot come from
any term with $w\ne \id$. By Proposition~\ref{prop:bruhat} (applied
with $w = \id$ and $\sigma = a_\ui q_{\lambda(\ui)}$), we have
\begin{align*}
 &\quad\,\,\Upsilon( \sum_\ui a_\ui B^{\omega_{i_k}} \dotsm B^{\omega_{i_1}}
  t_{\la+\la(\ui)}) \\
  &= \Upsilon( \sum_\ui a_\ui B^{\omega_{i_k}} \dotsm B^{\omega_{i_1}}
  \Theta^\la_\id(q_{\la(\ui)}) ) \\
  &= \Upsilon( \Theta^\la_\id ( \sum_\ui a_\ui q_{\la(\ui)} *
  [\omega_{i_1}] * \dotsm * [\omega_{i_k}] )) \\
  &= \Upsilon( \Theta^\lambda_\id(\sigma^w)) = A_{wt_{\lambda}}.
\end{align*}
where we have used our assumption that ${\mathfrak S}_w$ represents
the class $\sigma^w$.
\end{proof}

\begin{remark}\begin{enumerate}
\item
In Theorem~\ref{thm:schub} (and many other places in the paper) it
is possible to use the operators $C^\mu$ instead of $B^\mu$ to
obtain similar results.
\item Theorem~\ref{thm:schub} can be evaluated at 0 via $\phi_0:S \to
\Z$ to give a formula for $\phi_0(j(\xi_{wt_\lambda}))$ in terms of
non-equivariant quantum Schubert polynomial.  The elements
$\phi_0(j(\xi_{wt_\lambda}))$ lie inside what is called the {\it
affine Fomin-Stanley subalgebra} in \cite{Lam}, and are related to
the theory of affine Stanley symmetric functions. See~\cite{FGP,Mar}
for discussions on how to produce (non-equivariant) quantum Schubert
polynomials.
\end{enumerate}
\end{remark}

Let us call $a = \sum_{x \in W_\af} a_x \, A_x \in \A_\af$ {\it
superregular} if $a_x = 0$ for all $x \in W_\af \setminus
W_\af^\sreg$.
\begin{cor}\label{cor:span}
The elements $b(\lambda;\mu^1,\mu^2,\ldots,\mu^k)$ span the set of
superregular elements of $Z_{\A_\af}(S)$.
\end{cor}
\begin{proof}
Let $a$ be a superregular element in $Z_{\A_\af}(S)$.  By
Theorem~\ref{thm:pet}, it is an $S$-linear combination of
$j(\xi_{x})$ where $x$ is superregular.  By Theorem~\ref{thm:schub},
$j(\xi_{x})$ lies in the span of the elements
$b(\lambda;\mu^1,\mu^2,\ldots,\mu^k)$.
\end{proof}

\begin{cor}
Let $\mu \in P$ be an integral weight and $\Upsilon(f) \in
Z_{\A_\af}(S)$ for a sufficiently superregular $f$.  Then
$\Upsilon(B^\mu (f)) \in Z_{\A_\af}(S)$.
\end{cor}
\begin{proof}
Follows immediately from Corollary~\ref{cor:span} and
Theorem~\ref{thm:central}.
\end{proof}

The following Proposition is contained in \cite[Proposition
4.5]{Lam}.
\begin{prop}\label{prop:jt}
Let $\la\in \tQ$. Then $j(\xi_{t_\la}) = \sum_{\mu\in W\cdot\la}
A_{t_\mu}$.
\end{prop} The superregular case of
Proposition~\ref{prop:jt} follows from Theorem~\ref{thm:schub}.  The
general case can be obtained by a direct calculation, similar to
(but simpler than) Proposition~\ref{prop:superregular}.

\begin{prop} \label{prop:jst}
Let $\la$ be superregular antidominant. Then
\begin{align*}
&j(\xi_{r_it_\la}) = b(\lambda;\omega_i) = \Upsilon(B^{\omega_i}
\sum_{w\in
W}t_{w\lambda}) \\
&=\sum_{w \in W} \left((\omega_i - w\omega_i)A_{t_{w\lambda}} +
\sum_{\alpha \in R^+}
\ip{\alpha^\vee,\omega_i}A_{r_{w\alpha}t_{w\lambda}} + \sum_{\alpha
\in R^+} \ip{\alpha^\vee,\omega_i} A_{r_{w\alpha} t_{w(\lambda +
\alpha^\vee)}}\right)
\end{align*}
 where $\omega_i$ denotes the $i$-th fundamental
weight, and the two inner summations are as in
Theorem~\ref{thm:monk}.
\end{prop}
\begin{proof}
This follows immediately from Theorem~\ref{thm:schub} and the fact
that $\sigma^{r_i} = [\omega_i]$ in $QH^T(G/B)$.
\end{proof}

\section{Borel case}
\begin{prop}
\label{prop:HTmult} Let $x \in W_\af^-$ and $\lambda \in \tQ$. Then
$$
\xi_x \, \xi_{t_\la} = \xi_{xt_\la}.
$$
\end{prop}
\begin{proof}
By Lemma~\ref{lem:wla}, $\ell(x) + \ell(t_\la) = \ell(xt_\la)$.  The
proposition follows immediately from Proposition~\ref{prop:jt} and
Theorem~\ref{thm:j}.
\end{proof}

In particular $\{\xi_{t_\lambda} \mid \lambda \in \tQ\}$ is a
multiplicatively closed set that contains no zero divisors. So it
makes sense to consider $H^t_T(\Gr_G) = H_T(\Gr_G)[\xi_t^{-1} \mid t
\in \tQ]$. Let us also define $QH^T_q(G/B) = QH^T(G/B)[q_i^{-1} \mid
i \in I]$ and for $\alpha^\vee = \sum_i a_i \alpha_i^\vee\in Q^\vee$
we write $q_{\alpha^\vee}:= \prod_{i \in I}q_i^{a_i}$.

Let $\psi: H^t_T(\Gr_G) \to QH^T_q(G/B)$ be the $S$-module
homomorphism defined by
$$
\xi_{wt_\lambda} \, \xi_{t_\mu}^{-1} \mapsto q_{\lambda - \mu}
\sigma^w.
$$
This map is well-defined by Proposition~\ref{prop:HTmult} and is
clearly an $S$-module isomorphism.  Our main theorem is the
following.

\begin{thm} \label{thm:main}
The map $\psi: H^t_T(\Gr_G) \to QH_q^T(G/B)$ is an $S$-algebra
isomorphism.
\end{thm}

\begin{proof}
It is enough to show that $H^t_T(\Gr_G)$ satisfies the
$\psi$-preimage of the quantum equivariant Chevalley formula
(Theorem~\ref{thm:monk}) since this completely determines
$QH^T(G/B)$ (see \cite{Mih}).  By Proposition~\ref{prop:HTmult}, it
is enough to calculate the product $\xi_{r_it_\lambda} \,\xi_{w
t_\mu}$ in $H_T(\Gr_G)$ for superregular antidominant $\lambda,\mu
\in \tQ$.  One does so using Proposition~\ref{prop:jst} and
Theorem~\ref{thm:j}.  For each term $A_{r_{w \alpha} t_{w \lambda}}$
in $j(r_i t_\lambda)$ one obtains a term $\xi_{wr_\alpha
t_{\lambda+\mu}}$ in the product $\xi_{r_it_\lambda} \,\xi_{w
t_\mu}$ since
$$
(r_{w \alpha} t_{w \lambda})\; (w t_\mu) = w r_{\alpha}w^{-1} w
t_{\lambda} w^{-1} w t_\mu = wr_\alpha t_{\lambda+\mu}
$$
is a length-additive product and $wr_\alpha t_{\lambda+\mu} \in
W_\af^-$ by Lemma~\ref{lem:grass}.  The analogous statement holds
for terms of the form $r_{w \alpha} t_{w ( \lambda + \alpha^\vee)}$,
thus ensuring that the the product $\xi_{r_it_\lambda} \,\xi_{w
t_\mu}$ contains terms of the form $ \xi_{wr_{\alpha} t_{\mu +
\lambda + \alpha^\vee}}$ where $\ell(wr_\alpha) = \ell(w) -
2\inner{\alpha^\vee}{\rho} + 1$.  Furthermore, for $v\ne w$,
$(r_{v\alpha}t_{v\lambda})\; (wt_\mu)$ is never a length-additive
product since $\ell(t_{w^{-1}v\lambda + \mu}) \ll \ell(t_\lambda) +
\ell(t_{\mu})$. Similarly $(r_{v\alpha}t_{v(\lambda+\al^\vee)})\;
(wt_\mu)$ is never a length-additive product. Similar computations
hold for the equivariant terms $(\om_i-w\om_i) A_{t_{w\la}}$ in
$j(r_it_\la)$. Thus
$$
\xi_{r_it_\lambda} \,\xi_{w t_\mu} = (\omega_i - w\omega_i)
\xi_{wt_{\lambda+\mu}} + \sum_{\alpha \in R^+}
\ip{\alpha^\vee,\omega_i}\xi_{wr_{\alpha}t_{\lambda+\mu}} +
\sum_{\alpha \in R^+} \ip{\alpha^\vee,\omega_i} \xi_{wr_{\alpha}
t_{\mu + \lambda + \alpha^\vee}}
$$
where the two summations are exactly as in Theorem~\ref{thm:monk}.
Applying $\psi$ gives exactly Theorem~\ref{thm:monk} with both sides
multiplied by $q_{\lambda+\mu}$.
\end{proof}

The following corollary writes the equivariant Gromov-Witten
invariants of $G/B$ in terms of Schubert structure constants of
$H_T(\Gr_G)$.
\begin{cor}
Let $w,v,u \in W$ and $\la \in Q^\vee$.  Then the equivariant three
point Gromov-Witten invariant $c_{w,v}^{u,\la}$ is equal to the
coefficient of $\xi_{z}$ in the product $\xi_{x}\,\xi_{y} \in
H_T(\Gr_G)$, where $x = wt_\eta, y = vt_\kappa, z = ut_\mu \in
W_\af^-$ and $\la = \mu-\eta-\kappa$.
\end{cor}

Now we can write all the coefficients $j_x^y$ in terms of three
point genus zero Gromov-Witten invariants of $G/B$, and conversely.
%

\begin{thm}\label{thm:compare}
Let $x = wt_\lambda \in W_\af^-$ and $y = ut_\nu \in W_\af$ where we
assume $\nu \in Q^\vee$ is superregular.  Let $v\in W$ be the unique
element such that $v^{-1} \nu \in \tQ$. Then
$$
j_x^y = c_{w,v}^{uv,v^{-1}\nu - \lambda},
$$
where $c_{w,v}^{u,\kappa} = 0$ if $\kappa$ is not a nonnegative sum
of simple coroots.  Conversely, suppose $f,g,h \in W$ and $\eta \in
Q^\vee$ are given.  Then
$$
c_{f,g}^{h,\eta} = j_{ft_\la}^{hg^{-1}t_{g(\eta + \la)}}
$$
for sufficiently superregular antidominant $\la \in \tQ$.
\end{thm}
\begin{proof}
Let $z = vt_\mu \in W_\af^-$ where $\mu$ is chosen to be
superregular.  By Theorem \ref{thm:j}, we know that $j_x^y$ is the
coefficient of $\xi_{yz} = \xi_{uvt_{v^{-1}\nu + \mu}}$ in
$\xi_{x}\xi_{z}$, as long as $yz \in W_\af^-$ and $\ell(yz)= \ell(y)
+ \ell(z)$.  Using Lemmata~\ref{lem:grass} and \ref{lem:wla}, we
check that the latter two conditions are immediate with our
assumptions.  Applying the map $\psi$ of Theorem~\ref{thm:main}, we
see that $j_x^y$ is equal to the coefficient of $q_{v^{-1}\nu +
\mu}\sigma^{uv}$ in $q_{\mu+\la}\sigma^w\sigma^v$.  To obtain the
second statement from the first, it suffices to note that $\eta+\la$
is superregular antidominant if $\la \in \tQ$ is sufficiently
superregular.
\end{proof}

\begin{remark}Theorem~\ref{thm:compare} only writes $j_x^y$ for
superregular $y\in W_\af$ in terms of Gromov Witten invariants of
$G/B$.  To obtain the rest of the $j$-coefficients, one can use
Proposition~\ref{prop:jt} and the observation that for any $y \in
W_\af$ there is a length additive product $y t_\mu$ (with $\mu \in
Q^\vee$) which is superregular.
\end{remark}

Mihalcea~\cite{Mih} has shown that equivariant Gromov-Witten
invariants are polynomials in simple roots with nonnegative
coefficients (in fact Mihalcea uses negative simple roots).  As a
consequence we obtain a positivity result for the $j$-coefficients,
and hence for all affine homology structure constants of
$H_T(\Gr_G)$.
\begin{cor}\label{cor:pos}
All equivariant homology Schubert structure constants of
$H_T(\Gr_G)$ are nonegative polynomials in the simple roots.  For
each $x \in W_\af^-$ and $y \in W_\af$, the polynomial $j_x^y \in S$
is a nonnegative polynomial in the simple roots.
\end{cor}

It would be interesting to obtain a direct proof of
Corollary~\ref{cor:pos} which does not appeal to quantum cohomology,
even for the nonequivariant ($\ell(x) = \ell(y)$) case.


\section{Parabolic case}
Let $P\subset G$ be a standard parabolic subgroup. Following
Peterson, up to localization we show that $QH^T(G/P)$ is a quotient
of $H_T(\Gr_G)$.

\subsection{Extended affine Weyl group}
\label{SS:XWeyl}

Recall that $W$ acts on the coweight lattice $P^\vee$. Therefore we
may define the extended affine Weyl group $\We \cong W \ltimes
P^\vee$; as before, the element in $\We$ corresponding to $\la\in
P^\vee$ is denoted $t_\la$. $\We$ acts on the affine root lattice
$Q_\af$ by the same formula as \eqref{E:WaffQaff} with $\la\in
P^\vee$. There is an induced action of $\We$ on $Q\cong
Q_\af/\Z\delta$.

$\We$ is not a Coxeter group. However it still permutes $R_\af^\re$,
so for $x\in \We$ we can define its inversion set $\Inv(x)$ and
length $\ell(x)$ in the same way as for $x\in W_\af$. The set
$\We^0=\{\tau\in \We\mid \ell(\tau)=0\}$ of elements of $\We$ of
length zero, forms a subgroup of $\We$ since it is the stabilizer of
the set $R_\af^+$.

Let $\delta = \al_0+\theta = \sum_{i\in I_\af} a_i \al_i$ be the
null root. A node $i\in I_\af$ is called \textit{special} if
$a_i=1$, or equivalently, if there is an automorphism of the affine
Dynkin diagram taking the node $i$ to the Kac $0$ node. Denote by
$I^s\subset I_\af$ the set of special nodes. The nodes in
$I^s\setminus\{0\}$ are also called \textit{cominuscule}. The
abelian group $\Sigma = P^\vee/Q^\vee$ consists of the elements
$\om_i^\vee+Q^\vee$ for $i\in I^s$ where $\om_0^\vee=0$.

There is an isomorphism $\Sigma \cong \We^0$ which can be described
as follows. Let $i\in I^s$. Addition by the element
$-\om_i^\vee+Q^\vee \in \Sigma$, defines a permutation of the
elements of $P^\vee/Q^\vee$ or equivalently, a permutation of the
set $I^s$. This permutation extends uniquely to an automorphism
$\tau_i$ of the affine Dynkin diagram and satisfies $\tau_i(i)=0$.
It acts on $Q_\af$ by $\tau_i(\al_j)=\al_{\tau_i(j)}$ for all $j\in
I_\af$. It follows that $\tau_i(\delta)=\delta$ so that $\tau_i\in
\We^0$. The above isomorphism is given by
$-\om_i^\vee+Q^\vee\mapsto\tau_i$. Note that $\tau_0$ is the
identity in $\We$.

Define $v_i\in W$ to the shortest element such that $v_i \om_i = w_0
\om_i$ and let $\om_0=0$ so that $v_0=1$. Then
\begin{align} \label{E:special}
  \tau_i = v_i t_{-\om_i^\vee}.
\end{align}
Moreover $W^s=\{v_i\mid i\in I^s\}$ forms a subgroup of $W$ and the
map $\We^0\to W^s$ given by $\tau_i\mapsto v_i$, is an isomorphism.

%



\subsection{Affinization of $W_P$}
Let $L_P\subset G$ be the Levi factor of the parabolic subgroup
$P\subset G$. Say that $L_P$ has Dynkin node set $I_P$, root system
$R_P$, root lattice $Q_P$, coroot lattice $Q_P^\vee$, coweight
lattice $P_P^\vee$, and Weyl group $W_P$. Let $W^P$ denote the set
of minimal length coset representatives in $W/W_P$. Define
\begin{align} \label{E:WPaff}
(W_P)_\af &= W_P \ltimes Q_P^\vee = \{wt_\la\in W_\af\mid w \in W_P,
\la \in Q_P^\vee\}.
\end{align}

$L_P$ has affine root lattice $(Q_P)_\af = Q_P \oplus
\Z\delta\subset Q_\af$, affine Weyl group $(W_P)_\af$, and extended
affine Weyl group $\We_P = W_P \ltimes P_P^\vee$.

Let $I_P = \bigsqcup_{m=1}^k I_m$ be the partition of the node set
of $I_P$ according to the connected components of the subgraph of
the Dynkin diagram of $G$ induced by the subset of nodes $I_P$.
Write $R_m$, $Q_m^\vee$, $P_m^\vee$ for the irreducible
subrootsystem, coroot lattice and coweight lattice respectively.
Then we have an isomorphism of abelian groups
\begin{align} \label{E:Sigmaprod}
  \Sigma_P \cong \prod_{m=1}^k \Sigma_m
\end{align}
where $\Sigma_P = P_P^\vee/Q_P^\vee$ and $\Sigma_m =
P_m^\vee/Q_m^\vee$.

Let $(I_m)_\af = I_m \cup \{0_m\}$; the zero nodes for various $m$
are distinct. Write $(Q_m)_\af = Q_m \oplus \Z\delta\subset
(Q_P)_\af\subset Q_\af$. Let $\al_{0_m} = \delta - \theta_m\in
(Q_m)_\af$ where $\theta_m\in R_m^+$ is the highest root. Then
$(Q_m)_\af$ has basis $\{\al_i\mid i\in (I_m)_\af \}$. $\Sigma_m$
acts on $(Q_m)_\af$, inducing a permutation of $(I_m)_\af$ defined
by $\tau(\al_i)=\al_{\tau(i)}$ for $i\in (I_m)_\af$. Note that
$\Z\delta\subset (Q_m)_\af\subset Q_\af$ is fixed under the action
of $\Sigma_m$.



\subsection{$(W^P)_\af$}
Let
\begin{align*}
  (R_P)_\af^+ &= \{\beta\in R_\af^+\mid \bar{\beta}\in R_P
  \} \\
 (W^P)_\af &= \{x \in W_\af \mid x \cdot \beta >0 \text{ for
all $\beta\in (R_P)_\af^+$}\}.
\end{align*}

\begin{remark} \label{R:r0} Suppose $P\ne G$, or equivalently,
$\theta\not\in R_P$. Then $r_0\in (W^P)_\af$, since $r_0$ has the
lone inversion $\al_0=\delta-\theta\not\in(R_P)_\af^+$.
\end{remark}

\begin{lem} \label{L:WPaffchar} $wt_\la\in (W^P)_\af$ if and only
if, for every $\al\in R_P^+$, if $w\al>0$ then $\inner{\la}{\al}=0$
and if $w\al<0$ then $\inner{\la}{\al}=-1$.
\end{lem}
\begin{proof} For any $x\in W_\af$ and $\al\in R^+$, if $\al+n\delta\in
\Inv(x)$ for some $n\in\Z_{\ge0}$ then $\al\in \Inv(x)$. Similarly,
if $-\al+n\delta\in \Inv(x)$ for some $n\in\Z_{>0}$ then
$\delta-\al\in \Inv(x)$. Therefore $wt_\la\in (W^P)_\af$ if and only
if, for every $\al\in R_P^+$, $\al\not\in \Inv(wt_\la)$ and
$\delta-\al\not\in \Inv(wt_\la)$. The lemma follows
straightforwardly from these conditions.
\end{proof}

\begin{lem} \label{L:WPaffirr}
Suppose that $wt_\la\in (W^P)_\af$, $R_P$ is an irreducible root
system, $\inner{\la}{\al_j}\ne0$ for some $j\in I_P$, and $w=w_1w_2$
where $w_1\in W^P$ and $w_2\in W_P$. Then
\begin{enumerate}
\item
The node $j$ is cominuscule in $I_P$.
\item
For $\al\in R_P^+$,
\begin{align*}
  \inner{\la}{\al} = \begin{cases}
  -1 & \text{if $\al_j$ occurs in $\alpha$} \\
  0 & \text{otherwise.}
  \end{cases}
\end{align*}
\item
$w_2=v_j^P$, with notation as in Section \ref{SS:XWeyl}, with respect
to the cominuscule node $j$ in $I_P$.
\end{enumerate}
\end{lem}
\begin{proof}
We shall use Lemma \ref{L:WPaffchar} repeatedly without further
mention. We have $\inner{\la}{\al_j}=-1$. Suppose $\al_j$ occurs in
$\al\in R_P^+$, that is, $\al=\sum_{i\in I_P} a_i \al_i$ with
$a_j>0$ and all $a_i\ge0$. Then $\inner{\la}{\al}=\sum_{i\in I_P}
a_i \inner{\la}{\al_i} \le -a_j + \sum_{i\in I_P\setminus\{j\}} a_i
\inner{\la}{\al_i}\le -a_j \le -1$. Therefore $\inner{\la}{\al_i}=0$
for all $i\in I_P\setminus\{j\}$ and $a_j=1$. (1) and (2) follow.
For (3) we have $\Inv(w)\cap R_P^+ = \Inv(w_2)$. But $\Inv(w_2)$
must consist of the set of roots of $R_P^+$ in which $\al_j$ occurs.
This is precisely $\Inv(v_j^P)$. Hence $w_2=v_j^P$.
\end{proof}

\begin{lem} \label{L:WPafftrans} Suppose $wt_\la\in (W^P)_\af$ and
$w=w_1w_2\in W$ where $w_1\in W^P$ and $w_2\in W_P$. Then $w_2$ has
the following form. Let $J=\{j\in I_P\mid \inner{\la}{\al_j}=-1\}$.
Then $|J\cap I_m| \le 1$ for all $m$. If it is nonempty call this
element $j_m$; it is cominuscule in $I_m$. If it is empty write
$j_m=0_m\in (I_m)_\af$. Then $w_2 = \prod_{m=1}^k v_{j_m}^{I_m}$.
\end{lem}
\begin{proof} This follows from Lemma \ref{L:WPaffirr}.
\end{proof}

\begin{lem} \label{L:reflectP}
Let $\alpha\in R_\af^+$ be a real root. Then $r_\alpha\in (W_P)_\af$
if and only if $\bar\alpha\in R_P$.
\end{lem}
\begin{proof} Follows from \eqref{E:affinereflection}.
\end{proof}

\begin{lem} \label{L:funnyfact} \cite{Pet}
For every $w\in W_\af$ there is a unique factorization $w=w_1w_2$
for $w_1\in (W^P)_\af$ and $w_2\in (W_P)_\af$.
\end{lem}
\begin{proof} For existence we may assume that $w\alpha<0$ for some
$\alpha\in R_\af^+$ such that $\bar\alpha\in  R_P$. Then
$wr_\alpha<w$ and by Lemma \ref{L:reflectP} we have $r_\alpha\in
(W_P)_\af$. By induction $wr_\alpha=x_1 x_2$ with $x_1\in (W^P)_\af$
and $x_2\in (W_P)_\af$. Then $w=x_1(x_2r_\alpha)$ as desired.

For uniqueness, suppose $w=w_1w_2=w_1'w_2'$ with $w_1,w_1'\in
(W^P)_\af$ and $w_2,w_2'\in (W_P)_\af$. Then
$w_1w_2(w_2')^{-1}=w_1'\in (W^P)_\af$. Let $v=w_2(w_2')^{-1}\in
(W_P)_\af$. If $v\ne 1$ then there is some $\beta\in R_\af^+$ such
that $\overline{\beta}\in R_P$ and $v\beta<0$. But
$\overline{v\beta}\in R_P$. Since $w_1\in (W^P)_\af$, we have $w_1
v\cdot\beta<0$, contradicting the assumption that $w_1v=w_1'\in
(W^P)_\af$. Uniqueness follows.
\end{proof}

Define $\pi_P:W_\af\to (W^P)_\af$ by $w\mapsto w_1$ in the notation
of Lemma \ref{L:funnyfact}.

\begin{lem} \label{L:WPformula}
Let $\psi_P:Q^\vee\to P_P^\vee$ be the linear map defined by
\begin{align*}
  \psi_P(\la) = \sum_{j\in I_P} \inner{\la}{\al_j} \om_j^\vee.
\end{align*}
Let
\begin{align*}
  \psi_P(\la) + Q_P^\vee \mapsto
  (-\om_{j_1}^\vee+Q_1^\vee,\dotsc,-\om_{j_k}^\vee+Q_k^\vee)
\end{align*}
under the isomorphism \eqref{E:Sigmaprod} and define
\begin{align}
  \phi_P(\la) = -\psi_P(\la) - \sum_{m=1}^k \om_{j_m}^\vee \in Q_P^\vee.
\end{align}
Then
\begin{align*}
  \pi_P(t_\la) &= v t_{\la+\phi_P(\la)}
\end{align*}
where $v = \prod_{m=1}^k v_{j_m}^{I_m}\in W_P$.
\end{lem}
\begin{proof}
Since $\phi_P(\la)\in Q_P^\vee$, by definition
$\pi_P(t_\la)=\pi_P(t_{\la+\phi_P(\la)})$.  We have
\begin{align*}
  v(\la+\phi_P(\la))-(\la+\phi_P(\la)) &= \sum_m (\om_{j_m}^\vee - v_{j_m}^{I_m}
  \om_{j_m}^\vee) \in Q_P^\vee.
\end{align*}
Therefore $$\pi_P(t_{\la+\phi_P(\la)}) =
\pi_P(t_{v(\la+\phi_P(\la))})=\pi_P(v t_{\la+\phi_P(\la)} v^{-1}) =
\pi_P(v t_{\la+\phi_P(\la)}).$$ It suffices to show that
$vt_{\la+\phi_P(\la)}\in (W^P)_\af$. To this end, let $\alpha
+n\delta\in (R_P)_\af^+$. We have $\al\in R_p$ for some $1\le p\le
k$. Then
\begin{align*}
  v t_{\la+\phi_P(\la)} (\al+n\delta) &=
  v \al + (n + \sum_{m=1}^k \inner{\om_{j_m}^\vee}{\al})\delta \\
  &= v_{j_p}^{I_p}\al + (n+\inner{\om_{j_p}^\vee}{\al})\delta.
\end{align*}
If $j_p=0_p$ then $v_{j_p}^{I_p}=1$, $\om_{j_p}^\vee=0$, and $v
t_{\la+\phi_P(\la)} (\al+n\delta)=\al+n\delta\in R_\af^+$. If
$j_p\ne 0_p$ then $\inner{\om_{j_p}^\vee}{\al}=1$ and
$vt_{\la+\phi_P(\la)} (\al+n\delta)=v_{j_p}^{I_p}\al+(n+1)\delta\in
R_\af^+$ as desired.
\end{proof}


\begin{lem}\label{lem:positive}
Suppose $\la \in \tQ$ is antidominant.  Then $\phi_P(\la)$ is a
non-negative sum of positive coroots $\{\alpha^\vee_i \mid i \in
I_P\}$.
\end{lem}
\begin{proof}
We may suppose $P$ is irreducible.  If $\la \in \tQ$ then $\mu =
-\psi_P(\la) \in P^\vee_P$ is a dominant coweight.  But it is well
known (see~\cite[Section 13]{Hum}) that $\mu - \omega^\vee_i$ is a
sum of positive coroots for some cominuscule node $i \in I^s_P$.
\end{proof}

\begin{ex} We compute some examples of $\pi_P(t_\la)$ using Lemma
\ref{L:WPformula}, working within the subsystem $R_P$.
\begin{enumerate}
\item
In type $A_3$ let $I_P=\{2,3\}=I_1$ and $\la=-\al_1^\vee$. $R_P$ is
an irreducible subsystem of type $A_2$. We have $\psi_P(-\al_1^\vee)
=\om_2^\vee\in P_P^\vee$ and $\om_2^\vee = -\om_3^\vee +
(\al_2^\vee+\al_3^\vee)$. Therefore $j_1=3$, $v_3=r_2r_3$,
$\phi_P(-\al_1^\vee)=-\al_2^\vee-\al_3^\vee$, and
$\pi_P(t_{-\al_1^\vee}) = r_2 r_3
t_{-\al_1^\vee-\al_2^\vee-\al_3^\vee}=r_2r_3t_{-\theta^\vee}$ where
$\theta^\vee$ is the coroot associated to the highest root $\theta$.

Doing this another way, we have $-\al_1^\vee = -r_2r_3\theta^\vee$,
so that $t_{-\al_1^\vee} = r_2 r_3 t_{-\theta^\vee} r_3 r_2$.
Removing the right factor $r_3r_2\in (W_P)_\af$ we obtain
$r_2r_3t_{-\theta^\vee}$.
\item
In type $A_3$ let $I_P=\{1,3\}$ and $\la=-\al_2^\vee$. Then
$I_P=I_1\sqcup I_2$ with $I_1=\{1\}$ and $I_2=\{3\}$ with $R_1$ and
$R_2$ both of type $A_1$. We have
$\psi_P(-\al_2^\vee)=\om_1^\vee+\om_3^\vee\in P_P^\vee$. We have
$\om_1^\vee=-\om_1^\vee+\al_1^\vee$ and $v_1=r_1$ in $R_1$ and
$\om_3^\vee=-\om_3^\vee+\al_3^\vee$ and $v_3=r_3$ in $R_2$.
Therefore $\phi_P(\la)=-\al_1^\vee-\al_3^\vee$ and
$\pi_P(t_{-\al_2^\vee}) = r_1 r_3
t_{-\al_1^\vee-\al_2^\vee-\al_3^\vee}=r_1 r_3 t_{-\theta^\vee}$.

Another way, we have $t_{-\al_2^\vee} = r_1 r_3 t_{-\theta^\vee} r_3
r_1$, and removing the right factor $r_3 r_1\in (W_P)_\af$ we have
$r_1 r_3 t_{-\theta^\vee}$ as desired.
\item
In type $C_3$ with $\alpha_3$ the long root, let $I_P=\{2,3\}=I_1$
so that $R_P$ is an irreducible subsystem of type $C_2$. Let
$\la=-\al_1^\vee$. Then $\psi_P(-\al_1^\vee)=\om_2^\vee$. But in
$R_P$ we have $\om_2^\vee=\al_2^\vee+\al_3^\vee$. In particular
$j_1=0$ and $\pi_P(t_{-\al_1^\vee}) =
t_{-\al_1^\vee-\al_2^\vee-\al_3^\vee} = t_{-\theta^\vee}$.

Another way, we have $-\al_1^\vee = -\theta^\vee + r_1 \theta^\vee$.
Therefore we get $t_{-\al_1^\vee} = t_{-\theta^\vee} r_1
t_{\theta^\vee} r_1 = r_\theta r_0 r_1 r_0 r_\theta r_1 =
r_{(12321)010(12321)1}=r_{1232010232}$.  Because $r_2r_3r_2\in
(W_P)_\af$ we can remove this right factor. $r_{1232010}$ has
inversion $\delta-2\alpha_2-\alpha_3=\al_0+2\al_1=r_1(\al_0)$, so
$r_{101}=r_{\delta-2\al_2-\al_3}$ and
$r_{1232010}r_{101}=r_{123210}=t_{-\theta^\vee}$ as desired.
\item
In type $B_3$ with $\al_3$ the short root, let $I_P=\{2,3\}=I_1$ so
that $R_P$ is irreducible of type $B_2$. Let $\la=-\al_1^\vee$. We
have $\psi_P(-\al_1^\vee)=\om_2^\vee$. We have $\om_2^\vee =
-\om_2^\vee + 2\al_2^\vee+\al_3^\vee$. Therefore $j_1=2$,
$v_1=r_2r_3r_2$, $\phi_P(\la)=-2\al_2^\vee-\al_3^\vee$, and
$\pi_P(t_{-\al_1^\vee}) = r_2r_3r_2
t_{-\al_1^\vee-2\al_2^\vee-\al_3^\vee} = r_2r_3r_2
t_{-\theta^\vee}$.

Another way, $-\al_1^\vee = - r_2r_3r_2 \theta^\vee$, so
$t_{-\al_1^\vee} = r_2 r_3 r_2 t_{-\theta^\vee} r_2 r_3 r_2$.
Removing the right factor $r_2r_3r_2\in (W_P)_\af$ we obtain
$r_2r_3r_2t_{-\theta^\vee}$ as desired.
\end{enumerate}
\end{ex}

\begin{prop} \label{P:pip} \cite{Pet} Let $z\in W_\af$, $\beta\in
R^+_\af$, and $\la\in Q^\vee$.
\begin{enumerate}
\item \label{I:sub}
$\pi_P(W) \subset W^P \subset (W^P)_\af \subset (W_\af)^P$ where
$(W_\af)^P$ is the set of minimum length coset representatives for
$W_\af/W_P$.
\item \label{I:gr}
$\pi_P(W_\af^-)\subset W_\af^-$.
\item \label{I:le}
$\pi_P(z)\le z$.
\item \label{I:t}
$\pi_P(zt_\la)=\pi_P(z)\pi_P(t_\la)$. %
\end{enumerate}
\end{prop}
\begin{proof} \eqref{I:sub} follows from the definitions.
\eqref{I:le} follows from the proof of Lemma \ref{L:funnyfact}.

We first check  \eqref{I:t} for $z\in W$. Note that $\pi_P(z)=z_1$
where $z=z_1z_2$ is such that $z_1\in W^P$ and $z_2\in W_P$. We have
$zt_\la=z_1 t_{z_2\cdot\la} z_2$. Since $z_2\in W_P$ we have
$\la-z_2\cdot\la\in Q_P^\vee$. It follows that
$\pi_P(zt_\la)=\pi_P(z_1t_\la)$. But $\pi_P(t_\la)$ stabilizes
$(R_P)_\af^+$ by the proof of Lemma \ref{L:WPformula}, and $z_1\in
W^P$ has no inversions in $(R_P)_\af^+$. Therefore
$z_1\pi_P(t_\la)\in (W^P)_\af$, which finishes the proof of
\eqref{I:t} for $z\in W$. Using this we may reduce the proof of
\eqref{I:t} for $z\in W_\af$, to the case that $z=t_{\la'}$ for some
$\la'\in Q^\vee$. Since $\pi_P(t_\la)$ stabilizes $(R_P)_\af^+$ it
follows that $\pi_P(t_{\la'})\pi_P(t_\la)\in (W^P)_\af$. Therefore
it is enough to show that $\pi_P(t_{\la'+\la})$ and
$\pi_P(t_{\la'})\pi_P(t_\la)$ differ by a right multiple of $t_\mu$
for some $\mu\in Q_P^\vee$. By Lemma \ref{L:WPformula} there exist
$v',v''\in W_P$ such that $\pi_P(t_{\la'})=v' t_{\la'+\phi_P(\la')}$
and $\pi_P(t_{\la'+\la})=v'' t_{\la'+\la+\phi_P(\la'+\la)}$. We have
\begin{align*}
  \pi_P(t_{\la'})\pi_P(t_\la) &= v' t_{\la'+\phi_P(\la')} v
  t_{\la+\phi_P(\la)} \\
  &= v' v t_{v(\la'+\phi_P(\la'))+\la+\phi_P(\la)}.
\end{align*}
But the map $Q^\vee\to W_P$ given by $\la\mapsto v$, where $v\in
W_P$ is such that $\pi_P(t_\la)=v t_{\la+\phi_P(\la)}$, is a group
homomorphism, that is, $v'' = v' v$. Moreover $\la'+\phi_P(\la')$
and its image under $v\in W_P$, differ by an element of $Q_P^\vee$.
Therefore \eqref{I:t} follows.

For \eqref{I:gr}, let $x = wt_\la \in W_\af^-$ for $\la \in \tQ$.
Then $\pi_P(t_\la) = vt_{\la +\phi_P(\la)}$ and $\pi_P(x) =
\pi_P(w)\pi_P(t_\la)$.  To show that $\pi_P(x) \in W_\af^-$ we check
that $\pi_P(x) \cdot \alpha_i > 0$ for each $i \in I$.  We will
repeatedly use the following criterion: $ut_\mu \cdot \al_i > 0$ if
and only if either $\inner{\mu}{\al_i} < 0$ or $\inner{\mu}{\al_i} =
0$ and $\al_i \notin \Inv(u)$. In particular we need to establish
one of these conditions for $u=\pi_P(w)v$ and $\mu=\la+\phi_P(\la)$.

Suppose first that $i \in I_P$.  Then by Lemma \ref{L:WPaffchar},
$\inner{\la+\phi_P(\la)}{\al_i} = -1$ or $0$ and in the case of $0$
we have $\al_i \notin \Inv(v)$ and thus $\al_i \notin
\Inv(\pi_P(w)v)$. In either case we are done.


Otherwise suppose that $i\not\in I_P$ and that the Dynkin node $i$
is not connected to any node in $I_P$.  Then
$\inner{\la+\phi_P(\la)}{\al_i} = \inner{\la}{\al_i}$ and $\al_i \in
\Inv(w) \Leftrightarrow \al_i \in \Inv(\pi_P(w))$.  Since $x \cdot
\al_i > 0$ we conclude that $\pi_P(x) \cdot \al_i > 0$.

Finally suppose that $i \notin I_P$ and that the set $J$ of nodes in
$I_P$ connected to $i$, is nonempty. By Lemma \ref{lem:positive},
$\inner{\la+\phi_P(\la)}{\al_i} \leq \inner{\la}{\al_i}$.  We are
immediately done if $\inner{\la}{\al_i} < 0$ or $\inner{\la}{\al_i}
= 0$ and $\inner{\phi_P(\la)}{\al_i} < 0$. Suppose otherwise, so
that $\phi_P(\la)$ does not involve any roots $\alpha_j$ where $j
\in J$.

We know by Lemma \ref{L:WPaffchar} that
$\inner{\la+\phi_P(\la)}{\al_j} = -1$ or $0$.  Suppose for some $j
\in J$ that $\inner{\la+\phi_P(\la)}{\al_j} = 0$.  Then since
$\phi_P(\la)$ does not involve $\alpha_j$, we have
$\inner{\la}{\al_j} = 0 = \inner{\phi_P(\la)}{\al_j}$.  Let $P'$ be
such that $I_{P'}=I_P \setminus \{j\}$.  We may suppose inductively
that $\pi_{P'}(x) \in W_\af^-$.  We claim that $\pi_{P'}(x) =
\pi_P(x)$. Since $(W_{P'})_\af \subset (W_P)_\af$ it suffices to
show that $\pi_{P'}(x) \in (W^P)_\af$.  We first note that by our
assumptions $\phi_P(\la) = \phi_{P'}(\la)$ (using the fact that a
cominuscule node in a component of $I_P$ is still cominuscule in
$I_{P'}$).  Let $\pi_{P'}(x) = ut_{\la + \phi_P(\la)}$.  Since
$\inner{\la +\phi_{P'}(\la)}{\al_j} = 0$ and $\pi_{P'}(x) \in
W_\af^-$ we have $u \cdot \al_j > 0$.  We can thus deduce using
Lemmata \ref{L:WPaffchar} and \ref{L:WPaffirr} that $\pi_{P'}(x) \in
(W^P)_\af$.

Thus we may assume for our chosen $i \in I_P$ (with
$\inner{\la}{\al_i} = 0$) that all $j \in J$ satisfy
$\inner{\la+\phi_P(\la)}{\al_j} = -1$.  Note that these $j$ all lie
in different connected components of $I_P$ (thus $|J| \in
\{1,2,3\}$).  We need to show that $\pi_P(w)v \cdot \alpha_i > 0$.
We may assume that $I_P$ is exactly the union of the connected
components $I_{P_j} \subset I_P$ containing each $j \in J$, so that
$v = \prod_{j \in J} v_j$ where $v_j \in W_{P_j}$ are the elements
described in Lemma \ref{L:WPaffirr}.  For each parabolic subgroup
$W_Q \subset W$, write $w_Q \in W_Q$ for its longest element.  Then
by definition $v_j = w_{P_j} w_{P'_{j}}$ where $P'_{j} = P_j
\setminus \{j\}$.  Also factorize $\pi_P(w)$ as $u'u$ where $u$ lies
in the parabolic subgroup $W' \subset W$ corresponding to the nodes
$\{i\} \cup I_P$ and $u'$ is minimal length in $W/W'$.  It suffices
to show that $uv \cdot \al_i > 0$.  We calculate that
$$
uv \cdot \al_i = u \prod_{j} w_{P_j} w_{P'_j} \cdot \al_i = u w_P
\cdot \al_i.
$$
But $u \in (W')^P$ so that $uw_P$ is a length-additive factorization
as $u \in (W')^P$ and $w_P \in (W')_P = W_P$.  We know $w_P \cdot
\al_k < 0$ for $k \in I_P$.  If $uw_P \cdot \al_i < 0$ as well then
we must have $uw_P = w'_0$, the longest element in $W'$.  But $w$
factorizes uniquely (and length-additively) as $u'( u u '')$ where
$u'' \in W_P$.  If $u = w'_0 w_P$ then $uu'' \cdot \al_i < 0$ which
in turn means $w \cdot \al_i < 0$, contradicting the assumption that
$x = wt_\la \in W_\af^-$.

%
%
\end{proof}

\subsection{Ideals of $H_T(\Gr_G)$}
\begin{prop}[\cite{Pet}]
For $\alpha \in R_\af^+$, the $S$-submodule $$K(\alpha) =
\bigoplus_{\substack{x \in W_\af^- \\ x \cdot \alpha < 0}} S \,
\xi_x$$ of $H_T(\Gr_G)$, is an ideal of $H_T(\Gr_G)$.
\end{prop}
\begin{proof} By \eqref{E:jaction} it suffices to show that
$K(\alpha)$ has a left $\A_\af$-action. By \eqref{E:NilHeckeOnHom}
it suffices to show that if $x\in W_\af^-$, $r_ix>x$, and
$x\alpha<0$, then $r_ix\alpha<0$.  Suppose not, that is,
$r_ix\alpha>0$. Then $x\alpha=-\al_i$ and
$0>-\alpha=x^{-1}\alpha_i$. But $x^{-1}<x^{-1}r_i$ so that
$x^{-1}\alpha_i>0$, a contradiction.
\end{proof}

Thus $$J_P = \sum_{\alpha \in (R_P)_\af^+} K(\alpha) = \sum_{x\in
W_\af^-\setminus (W^P)_\af} S \xi_x
$$ is an ideal of $H_T(\Gr_G)$.

\subsection{Parabolic quantum parameters}

\begin{lem}\label{L:nilHeckepiP}
Let $\la \in \tQ$.  Then $A_i \cdot \xi_{\pi_P(t_\la)} = 0 \mod J_P$
for each $i \in I$.
\end{lem}
\begin{proof}
By \eqref{E:NilHeckeOnHom} $A_i \cdot \xi_{\pi_P(t_\la)} = 0$ unless
$\ell(r_i \, \pi_P(t_\la)) = \ell(\pi_P(t_\la)) + 1$ and $r_i\,
\pi_P(t_\la)\in W_\af^-$.  By Lemma~\ref{L:WPformula}, $\pi_P(t_\la)
= vt_\nu$ for some $v\in W_P$ and $\nu \in \tQ$.

Suppose $i \notin I_P$.  Then $\ell(r_i v) = \ell(v) + 1$ and by
Lemma~\ref{lem:grass} $\ell(r_i v t_\nu) = \ell(vt_\nu) - 1$, so
$A_i \cdot \xi_{\pi_P(t_\la)} = 0$.

Suppose $i \in I_P$.  Then $r_i v \in W_P$. By
Lemma~\ref{L:WPafftrans} we have $r_i vt_\nu \notin (W^P)_\af$ and
$\xi_{\pi_P(t_\la)} = 0 \mod J_P$.
\end{proof}

Note that we exclude $i = 0$ in Lemma~\ref{L:nilHeckepiP}.  The
following result generalizes Proposition~\ref{prop:HTmult}.

\begin{prop}\label{prop:Pfact}
Let $x \in W_\af^- \cap (W^P)_\af$ and $\lambda \in \tQ$.  Then
$x\pi_P(t_\lambda)\in W_\af^- \cap (W^P)_\af $ and we have
$$
\xi_x \, \xi_{\pi_P(t_\lambda)} = \xi_{x\pi_P(t_\lambda)} \mod J_P.
$$

\end{prop}
\begin{proof}
By Lemma \ref{L:nilHeckepiP}, $J \cdot \xi_{\pi_P(t_\lambda)} = 0
\mod J_P$, where $J = \sum_{w \in W \setminus\{\id\}} \A_\af A_w$ as
in Theorem~\ref{thm:pet}.  By Theorem~\ref{thm:pet} we thus have
\begin{align*}
\xi_x \, \xi_{\pi_P(t_\lambda)}  = A_x \cdot \xi_{\pi_P(t_\lambda)}
\mod J_P.
\end{align*}
It suffices thus to show that the product $x \, \pi_P(t_\lambda)$ is
length-additive.  Since $x \in W_\af^- \cap (W^P)_\af$ using
Proposition~\ref{P:pip} we may write $x = w\pi_P(t_\nu)$ for $w \in
W^P$ and $\nu \in \tQ$. We have $\ell(w
\pi_P(t_\mu))=-\ell(w)+\ell(\pi_P(t_\mu))$ for every $\mu\in \tQ$
such that $wt_\mu\in W_\af^-$, so it suffices to show that
$\ell(\pi_P(t_{\nu+\la})) = \ell(\pi_P(t_\lambda)) +
\ell(\pi_P(t_\nu))$ for $\nu,\lambda \in \tQ$. By Lemma
\ref{L:WPformula} we may assume that $\nu, \lambda$ are chosen so
that $\pi_P(t_\nu) = v_\nu \, t_\nu$ and $\pi_P(t_\lambda) =
v_\lambda \, t_\la$. By Lemma~\ref{L:WPafftrans}, $\ell(v_\lambda) =
-\ip{\la,2\rho_P}$ where $2\rho_P = \sum_{\alpha \in R_P^+} \alpha$
and similarly for $v_\nu$.  Thus by Lemma~\ref{lem:grass},
$\ell(v_\la \, t_\la) = -\ip{\la,2(\rho -\rho_P)}$ and similarly for
$\nu$ and $\nu + \la$.

The last statement follows immediately from Proposition~\ref{P:pip}
since $\pi_P(x t_\lambda) = x \pi_P(t_\lambda)$.
\end{proof}

\subsection{Quantum parabolic Chevalley formula}
The equivariant quantum cohomology $QH^T(G/P)$ is the free $S[q_i
\mid i \in I\setminus I_P]$-module spanned by the equivariant
quantum Schubert classes $\{\sigma_P^w \mid w \in W^P\}$.  For
$\lambda = \sum_i a_i \alpha_i^\vee \in Q^\vee/Q_P^\vee$ with $a_i
\in \Z$ we let $q_\lambda = \prod_{i \in I\setminus I_P} q_i^{a_i}$.
The quantum multiplication of $QH^T(G/P)$ is denoted again with $*$.

Recall that for $w \in W$, if we write $w=w_1w_2$ with $w_1\in W^P$
and $w_2\in W_P$ then $w_1=\pi_P(w)$. Recall that $2\rho_P =
\sum_{\alpha \in R_P^+} \alpha$. Let $\eta_P:Q^\vee\to
Q^\vee/Q_P^\vee$ be the natural projection.

\begin{thm}[Quantum equivariant parabolic Chevalley formula \cite{Mih}]
\label{thm:Pmonk} Let $i \in I\setminus I_P$ and $w \in W^P$.  Then
we have in $QH^T(G/P)$
\begin{align*}
\sigma_P^{r_i} \,*\, \sigma_P^{w} &= (\omega_i - w\cdot \omega_i)
\sigma_P^w + \sum_\alpha \ip{\alpha^\vee,
\omega_i}\,\sigma_P^{w r_\alpha} \\
&+ \sum_\alpha
\ip{\alpha^\vee,\omega_i}\,q_{\eta_P(\alpha^\vee)}\,\sigma^{\pi_P(wr_\alpha)}
\end{align*}
where the first summation is over $\alpha \in R^+ \setminus R_P^+$
such that $wr_\alpha \gtrdot w$ and $wr_\al\in W^P$, and the second
summation is over $\alpha \in R^+ \setminus R_P^+$ such that
$\ell(\pi_P(wr_\alpha)) = \ell(w) + 1 - \ip{\alpha^\vee,
2(\rho-\rho_P)}$.
\end{thm}
Mihalcea~\cite{Mih} showed that the quantum equivariant parabolic
Chevalley formula completely determines the multiplication in
$QH^T(G/P)$.

We will use a special case of the Peterson-Woodward comparison
formula to clarify the second summation in Theorem~\ref{thm:Pmonk}.
For $u,v,w \in W^P$ and $\la\in Q^\vee/Q_P^\vee$ let
$d^{w,\lambda,P}_{u,v}$ denote the coefficient of $q_\lambda
\sigma^w_P$ in $\sigma^u_P * \sigma^v_P$, calculated in $QH^*(G/P)$.
We use $d^{w,\lambda,P}_{u,v}$ instead of $c^{w,\lambda,P}_{u,v}$
since Woodward's result is stated only for the non-equivariant
coefficients.

\begin{thm}[{\cite[Lemma 1, Thm. 2]{Woo}}]\label{thm:woo} \
\begin{enumerate}
\item
For every $\la_P \in Q^\vee/Q_P^\vee$ there exists a unique $\la_B
\in Q^\vee$ such that $\eta_P(\la_B)=\la_P$ and
$\inner{\la_B}{\al}\in \{0,-1\}$ for all $\al\in R_P^+$. Moreover if
$\inner{\la_P}{\al_i}\le 0$ for $i\in I\setminus I_P$ then
$\inner{\la_B}{\al_i}\le0$ for all $i\in I$.
\item For every $x,y,z \in W^P$ we have
$$
d^{z,\la_P,P}_{x,y} = d^{z\,w_P \, w_{P'},\la_B}_{x,y}
$$
where $w_P$ is the longest element in $W_P$ and $P' \subset P$ is
the standard parabolic subgroup of $P$ such that $I_{P'}=\{i\in
I_P\mid \inner{\la_B}{\al_i}=0\}$.
\end{enumerate}
\end{thm}

\begin{remark}
In \cite{Woo}, Theorem~\ref{thm:woo} is stated instead in terms of
the coefficients $\ip{x,y,w_0 \, z \, w_P}_{\lambda_P} =
d^{z,\lambda_B,P}_{xy}$. Since $w_B=\id$, our formulation is
recovered.
\end{remark}

\begin{remark}
In Theorem \ref{thm:woo}, $\la_B$ and $P'$ may be computed
explicitly. Given $\la_P\in Q^\vee/Q_P^\vee$, let $\la\in Q^\vee$ be
defined by $\la=\sum_{i\in I\setminus I_P} \inner{\la_P}{\om_i}
\al_i^\vee$; it clearly satisfies $\eta_P(\la)=\la_P$. Let
$\pi_P(t_\la)=vt_{\la+\phi_P(\la)}$ be as in Lemma
\ref{L:WPformula}. Then $\la_B = \la+\phi_P(\la)$,
$I_{P'}=I_P\setminus \{j_m \mid \text{$1\le m\le k$ and $j_m\ne
0_m$}\}$, and $v = w_P w_{P'}$.
\end{remark}

\begin{lem}
\label{lem:Pmonk} The second summation in Theorem~\ref{thm:Pmonk} is
over $\alpha \in R^+ \setminus R_P^+$ such that
\begin{enumerate}
\item
$\ell(\pi_P(wr_\alpha)) = \ell(w) + 1 - \ip{\alpha^\vee,
2(\rho-\rho_P)}$, and
\item
$\ell(wr_\alpha) = \ell(w) - \ip{\alpha^\vee,2\rho}  + 1$.
\end{enumerate}
\end{lem}
\begin{proof}
Using the notation of Theorems~\ref{thm:Pmonk} and \ref{thm:woo} set
$x= r_i$, $y = w$, $z = \pi_P(wr_\alpha)$, and
$\la_P=\eta_P(\al^\vee)$. Then the coefficient of
$q_{\eta_P(\al^\vee)}\sigma_P^{\pi_P(wr_\alpha)}$ in $\sigma^{r_i}_P
* \sigma^w_P$ is 0 unless the coefficient of
$q_{\la_B} \sigma^{\pi_P(wr_\alpha)\,w_P \, w_{P'}}$ in
$\sigma^{r_i}
* \sigma^w$ is non-zero.

By the Claim within Lemma 4.1 of \cite{FW}, we know that
$\pi_P(r_\alpha) \ne \pi_P(r_\beta)$ for any $\alpha \ne \beta$ both
in $R^+ \setminus R_P^+$.  Since $\pi_P(wr_\alpha)\,w_P \, w_{P'}
W_P = wr_\alpha W_P$ we conclude that the coefficient of
$\sigma_P^{\pi_P(wr_\alpha)}$ in $\sigma^{r_i}_P * \sigma^w_P$ is
non-zero only if $\sigma^{wr_\alpha}$ occurs in $\sigma^{r_i} *
\sigma^w$.  By Theorem~\ref{thm:monk} and the last statement of
Theorem~\ref{thm:woo}, the latter holds only if $\ell(wr_\alpha) =
\ell(w) - \ip{\alpha^\vee,2\rho}  + 1$.
\end{proof}

\begin{remark} Presumably Lemma \ref{lem:Pmonk} can be deduced from
Theorem~\ref{thm:Pmonk} purely Coxeter-theoretically; that is,
without the additional input provided by Theorem~\ref{thm:woo}.
\end{remark}

\subsection{Parabolic Peterson Theorem}
\begin{lem}\label{lem:bij}
The map $\pi_P(t_\nu) \mapsto \eta_P(\nu)$ is a bijection onto
$Q^\vee/Q^\vee_P$.
\end{lem}
\begin{proof}
By definition, $\pi_P(t_\nu) = \pi_P(t_{\nu+ \mu})$ if $\mu \in
Q^\vee_P$.  Thus the map is well defined and clearly a surjection.
By Proposition~\ref{P:pip}, it thus suffices to show that if
$\eta_P(\nu) = 0$ then $\pi_P(t_\nu) = \id$.  But $\eta_P(\nu) = 0$
means that $\nu \in Q^\vee_P$ so $t_\nu \in (W_P)_\af$ and
$\pi_P(t_\nu) = \id$.
\end{proof}
\begin{thm}\label{thm:Pmain}
There is an $S$-algebra isomorphism
\begin{align*}
 \Psi_P:(H_T(\Gr_G)/J_P)[\xi_{\pi_P(t_\la)}^{-1}\mid \la\in \tQ]
  &\longrightarrow QH^T(G/P)[q_i^{-1}\mid i\in I\setminus I_P] \\
\xi_{v\pi_P(t_\la)} \,\xi_{\pi_P(t_\nu)}^{-1} &\longmapsto
q_{\eta_P(\la-\nu)} \,\sigma^v_P
\end{align*}
for $v\in W^P$ and $\la,\nu\in \tQ$.
\end{thm}

\begin{proof}
Using Lemma~\ref{lem:bij}, the map $\Psi_P$ is easily seen to be an
isomorphism of $S$-modules. Since the quantum parabolic Chevalley
formula determines the ring structure of $QH^T(G/P)$, it suffices to
prove that the $\Psi_P$-preimage of this relation holds in
$H_T(\Gr_G)/J_P$. By Proposition~\ref{prop:Pfact}, it suffices to
check the product $\xi_{v \pi_P(t_\lambda)} \, \xi_{r_i
\pi_P(t_\nu)}$ for a choice of $\nu,\lambda \in \tQ$ for each $i\in
I\setminus I_P$ and $v \in W^P$.  Taking a large power of
$\pi_P(t_\lambda)$ and using Proposition~\ref{P:pip}, we may choose
$\nu,\lambda$ such that $\pi_P(t_\nu) = t_\nu$ and $\pi_P(t_\lambda)
= t_\lambda$. By Theorem~\ref{thm:main}, this reduces to checking
that the preimage (in the Borel case) of the quantum equivariant
Chevalley formula in $H_T(\Gr_G)$, gives rise to that of the quantum
equivariant parabolic Chevalley formula after quotienting out by the
ideal $J_P \subset H_T(\Gr_G)$.

The equivariant term and the non-quantum terms trivially agree, so
we check the quantum terms.  For $w \in W^P$ define
\begin{align*}
  A_w = \{\alpha \in R^+ \setminus R_P^+ \mid \ell(wr_\alpha) &=
\ell(w) - \ip{\alpha^\vee,2\rho}  + 1 \text{ and } \\
 \ell(\pi_P(wr_\alpha)) &= \ell(w) + 1 - \ip{\alpha^\vee,
 2\rho-2\rho_P} \}
\end{align*}
and
\begin{align*}
B_w = \{\alpha \in R^+ \setminus R_P^+ \mid \ell(wr_\alpha) &=
\ell(w) - \ip{\alpha^\vee,2\rho} + 1 \, \text{ and }
\\
\pi_P(wr_\alpha t_\alpha) &= wr_\alpha t_\alpha\}.
\end{align*}
Note that $A_w$ indexes quantum terms in the parabolic quantum
Chevalley formula by Lemma~\ref{lem:Pmonk} and $B_w$ indexes quantum
terms in the preimage of the quantum Borel Chevalley formula in
$H_T(\Gr_G)$ which do not vanish modulo $J_P$.

By Lemma~\ref{lem:BFP}, the condition $\ell(wr_\alpha) = \ell(w) -
\ip{\alpha^\vee,2\rho}  + 1$ implies that $\ell(wr_\alpha) = \ell(w)
- \ell(r_\alpha)$ and $\ell(r_\alpha) = \ip{\alpha^\vee,2\rho} -1$.
The equation $\ell(wr_\alpha) = \ell(w) - \ell(r_\alpha)$ in turn
implies that $r_\alpha \in W^P$, since $w \in W^P$.  Thus
$\ip{\alpha^\vee,\beta} \leq 0$ for $\beta \in R_P^+$.  Let $x =
wr_\alpha=yx'$ with $y =\pi_P(wr_\alpha) \in W^P$ and $x'\in W_P$.

Let $\al\in A_w$, that is, $\ell(y) = \ell(w) + 1 - \ip{\alpha^\vee,
2\rho-2\rho_P}$.  Thus $\ell(x) - \ell(y) = -\ip{\alpha^\vee,
2\rho_P} = -\sum_{\beta \in R_P^+} \ip{\alpha^\vee,\beta}$.  Let us
estimate $\ell(x') = |\Inv(x')|$. Since $xr_\alpha=w \in W^P$ we
must have $x' \beta > 0$ for $\beta \in R_P^+$ satisfying $r_\alpha
\beta = \beta$.  Hence $\ell(x) - \ell(y) = \ell(x') = |\Inv(x')|
\le | \{\beta\in R_P^+\mid \inner{\al^\vee}{\beta}<0\}| \le
-\ip{\alpha^\vee, 2\rho_P}$. Thus we must have $-1 \leq
\ip{\alpha^\vee,\beta} \leq 0$ for all $\beta\in R_P^+$ and
$\Inv(x') = \{\beta \in R_P^+ \mid \ip{\alpha^\vee,\beta} = -1\}$.
Using Lemma~\ref{L:WPafftrans}, we conclude that $x't_{\alpha^\vee}
= \pi_P(t_\alpha^\vee) \in (W^P)_\af$. This in turn gives $wr_\alpha
t_{\al^\vee} = x t_{\al^\vee} = y(x't_{\al^\vee})=
\pi_P(wr_\alpha)\pi_P(t_{\al^\vee})$, showing that $\alpha \in B_w$.

For the reverse inclusion $B_w \subset A_w$, one deduces from
$\pi_P(wr_\alpha t_\alpha) = wr_\alpha t_\alpha =
\pi_P(wr_\al)\pi_P(t_\alpha)$ that $x't_{\alpha} = \pi_P(t_\alpha)$
satisfies the conditions of Lemma~\ref{L:WPafftrans}. In particular
$\ip{\alpha^\vee,2\rho_P} = -\ell(x')$.  This shows that $B_w
\subset A_w$.

Finally we note that a term $q_{\alpha^\vee}\sigma^{wr_\alpha}$ (for
$w \in B_w$) in the quantum Borel Chevalley formula gives rise to
the class $\xi_{wr_\alpha}\xi_{t_{-\alpha^\vee}}^{-1} \in
H_T^t(\Gr_G)$ (where $\xi_{wr_\alpha} := \xi_{wr_\alpha t_\lambda}
\xi_{t_\lambda}^{-1}$ for appropriate $t_\lambda$) which in turn
gives rise to the class
$q_{\eta_P(\alpha^\vee)}\sigma_P^{\pi_P(wr_\alpha)}$ in $QH^T(G/P)$.
\end{proof}

For $w,v,u \in W^P$ and $\la \in Q^\vee/Q^\vee_P$ let
$c^{w,\lambda,P}_{u,v}$ denote the coefficient of $q_\lambda
\sigma^w_P$ in $\sigma^u_P * \sigma^v_P$, calculated in $QH^T(G/P)$.
\begin{cor}\label{cor:Pcompare}
Let $w,v,u \in W^P$ and $\la \in Q^\vee/Q^\vee_P$.  Pick
$\eta,\kappa,\mu \in \tQ$ so that $x = w\pi_P(t_\eta), y =
v\pi_P(t_\kappa), z = u\pi_P(t_\mu) \in W_\af^- \cap (W^P)_\af$,
where $\la=\eta_P(\mu - \eta+\kappa)$.  Then the equivariant three
point Gromov-Witten invariant $c_{u,v}^{w,\la,P}$ is equal to the
coefficient of $\xi_{z}$ in the product $\xi_{x}\,\xi_{y} \in
H_T(\Gr_G)$.
\end{cor}
Note that in Corollary~\ref{cor:Pcompare}, the element $z$ is
completely determined by $x, y$ and $\la$.


\begin{remark}
It would be interesting to compare Corollary~\ref{cor:Pcompare} with
the work of Buch, Kresch and Tamvakis \cite{BKT} who exhibit the
Gromov-Witten invariants of (classical, orthogonal and Lagrangian)
Grassmannians as classical Schubert structure constants.
\end{remark}

\section{Application to quantum cohomology}
For this section we will work in non-equivariant quantum cohomology
$QH^*(G/P)$ and homology $H_*(\Gr_G)$.
\subsection{Highest root}
We apply known formulae in $H_*(\Gr_G)$ to obtain new formulae in
$QH^*(G/P)$. Let $K=\sum_{i\in I_\af} a_i^\vee\al_i^\vee$ be the
canonical central element for the affine Lie algebra associated to
the Lie algebra of $G$. It satisfies $a_0^\vee=1$ and $\theta^\vee =
\sum_{i \in I} a_i^\vee \alpha_i^\vee$ where $\theta^\vee$ is the
coroot associated with the highest root $\theta$. Let $j_0$ denote
the composition of $j: H_*^T(\Gr_G) \to Z_\af(S)$ with the
evaluation $\phi$ at 0: $\phi(\sum_w a_w A_w) = \sum_w \phi_0(a_w)
A_w$, where $\phi_0: S \to \Z$ evaluates a polynomial at 0.

\begin{prop}[\cite{LS}] \label{prop:LS}
We have
$$
j_0(\xi_{r_0}) = \sum_{i \in I_\af} a_i^\vee A_i.
$$
Thus in $H_*(\Gr_G)$, for $x\in W_\af^-$ we have
$$
\xi_{r_0}\, \xi_{x} = \sum_{\substack{i \in I_\af
\\ r_ix>x \\ \ r_ix \in W_\af^-}} a_i^\vee \xi_{r_i
x}.
$$
\end{prop}

Suppose $P \neq G$.  By Remark \ref{R:r0}, $r_0 = r_\theta
t_{-\theta^\vee} \in (W^P)_\af$.

\begin{prop}\label{P:high}
Let $w \in W^P$.  We have
$$
\sigma_P^{\pi_P(r_\theta)} * \sigma_P^w =
q_{\eta_P(\theta^\vee-w^{-1}\theta^\vee)}\sigma^{\pi_P(r_\theta
w)}_P + q_{\eta_P(\theta^\vee)}\sum_{\substack{i \in I \\
r_iw<w }} a_i^\vee \,\sigma^{r_i w}_P
$$
where the first term is present if and only if $w \cdot \alpha =
\theta$ for some $\alpha \in R^+ \setminus R_P^+$.
\end{prop}
\begin{proof}
Let $x = w t_\la \in W_\af^- \cap (W^P)_\af$ where we assume as in
the proof of Theorem~\ref{thm:Pmain} that $\pi_P(t_\lambda) =
t_\lambda$.  By Lemma \ref{L:WPaffirr}, we have
$\ip{\lambda,\alpha_i} = 0$ for $i \in I_P$.  Using Lemma
\ref{L:WPformula}, we may assume in addition that
$\ip{\lambda,\alpha_i} \neq 0$ for $i \in I\setminus I_P$.  Thus by
Lemma~\ref{lem:grass}, we have $\ell(r_ix)= \ell(x) + 1$ and $r_ix
\in W_\af^-$ if and only if $\ell(r_iw) = \ell(w) - 1$ (which
automatically implies that $r_iw \in W^P$).

Now let us consider $r_0x = r_0wt_\lambda$.  By our assumptions,
$t_\lambda \cdot \alpha = \alpha$ for $\alpha \in R_P^+$, and since
the only inversion of $r_0$ is $\al_0=\delta - \theta$, we deduce
that $r_0x \in (W^P)_\af$ if and only if $w\alpha \neq \theta$ for
$\alpha \in R_P^+$.  If $r_0x \in (W^P)_\af$ then $r_0x = r_\theta
t_{-\theta^\vee}w t_\lambda = (r_\theta
w)t_{-w^{-1}\theta^\vee+\lambda} = \pi_P(r_\theta w)
\pi_P(t_{-w^{-1}\theta^\vee})\pi_P(t_\lambda)$ by
Proposition~\ref{P:pip}.

Also note that in the above situation,
\begin{align*}
\ell(r_0x) = \ell(x) + 1 &\Leftrightarrow r_0x > x \\
&\Leftrightarrow x \cdot (n\delta - \alpha) = \delta - \theta
&\text{for some $n\delta - \alpha \in R^+_\af$} \\ &\Leftrightarrow
w \cdot \alpha = \theta &\text{for some $\alpha \in R^+ \setminus
R_P^+$.}
\end{align*}
Finally, we observe that in the above situation we automatically
have $r_0x \in W_\af^-$ since $x \in W_\af^-$.

Using Proposition~\ref{prop:LS}, Theorem~\ref{thm:Pmain} and these
observations we obtain in $QH^*(G/P)$
\begin{align*}
q_{\eta_P(-\theta^\vee)}\sigma_P^{\pi_P(r_\theta)} *
q_{\eta_P(\lambda)} \sigma_P^w &= \sum_{\substack{i \in I
\\ r_i w < w }} a_i^\vee q_{\eta_P(\lambda)}\sigma^{r_i w}_P  \\
&+ a_0^\vee
q_{\eta_P(\lambda-w^{-1}\theta^\vee)}\sigma^{\pi_P(r_\theta w)}_P
\end{align*}
where the last term is present if and only if $w \cdot \alpha =
\theta$ for some $\alpha \in R^+ \setminus R_P^+$.  Dividing both
sides by $q_{\eta_P(\lambda-\theta^\vee)}$ and using $a_0^\vee = 1$,
we obtain the required statement.
\end{proof}

In the case that $P$ is a maximal parabolic corresponding to a
cominuscule node (as in the following section),
Proposition~\ref{P:high} looks similar to a formula shown to us by
Nicolas Perrin (see \cite{CMP}).

\subsection{Cominuscule case} In this section we assume that $P$ is
a maximal parabolic such that $I\setminus I_P = \{j\}$ where $j$ is
a cominuscule Dynkin node.

The map $W\to W$ given by $w\mapsto w^*=w_0 w w_0$, is an involutive
isomorphism that sends simple reflections to simple reflections:
$r_i\mapsto (r_i)^* = r_{i^*}$ for some $i^*\in I$. The map
$i\mapsto i^*$ is an automorphism of the finite Dynkin diagram.
There is an associated automorphism of $Q$ given by $\alpha\mapsto
\alpha^* := -w_0\alpha$ which satisfies $(\al_i)^* = \al_{i^*}$ for
$i\in I$. For $w\in W$ and $\al\in Q$ we have $(w\al)^*=w^* \al^*$.
There is a similar involution on $P^\vee$ that stabilizes $Q^\vee$,
thereby defining an involutive automorphism of
$\Sigma=P^\vee/Q^\vee$. Since $-w_0 \om_i^\vee = \om_{i^*}^\vee$ and
$\om_i^\vee \equiv w_0 \om_i^\vee \mod Q^\vee$, the induced
automorphism of $P^\vee/Q^\vee$ is given by negation:
$\om_i^\vee+Q^\vee\mapsto -\om_i^\vee+Q^\vee$.

The finite Dynkin automorphism $I\to I$ given by $i\mapsto i^*$, may
be extended to an automorphism of the affine Dynkin diagram by
letting $0^*=0$. This induces an automorphism of $W_\af$ again
denoted $w\mapsto w^*$.



\begin{prop} \label{P:section} Define
$\vartheta: W^P\to W_\af$ by $\vartheta(y)=\tau_j(y)^*$. Then for
every $y\in W^P$, $\vartheta(y)\in (W^P)_\af\cap W_\af^-$, and
$\{\xi_{\vartheta(y)}\mid y\in W^P\}$ is a
$S[\xi_{\pi_P(t_\la)}^{\pm}\mid \la\in \tQ]$-basis of
$(H_T(\Gr_G)/J_P)[\xi_{\pi_P(t_\la)}^{-1}\mid \la\in \tQ]$. Moreover
if $\vartheta(y)=wt_\la$ then $\pi_P(w)=\pi_P(w_0^P y)$.
\end{prop}
\begin{proof} Note that $i\mapsto \tau_j(i)^*=\tau_{j^*}(i^*)$ is
an involutive affine Dynkin automorphism that stabilizes
$I_\af\setminus\{0,j\}$ and exchanges $0$ and $j$. It follows that
$\al\mapsto \tau_j(\al)^*=\tau_{j^*}(\al^*)$ stabilizes $R_P^+$.
This map also permutes the affine simple roots and hence stabilizes
$R_\af^+$.

Let $y\in W^P$. Then $y\cdot \al_i >0$ for all $i\in
I_\af\setminus\{0,j\}$. Consequently $\vartheta(y)\cdot \al_i >0$
for all $i\in I_\af\setminus\{0,j\}$. Since $y\in W$, $\vartheta(y)$
is in the subgroup of $W_\af$ generated by $r_i$ for $i\in
I_\af\setminus \{j\}$, so that $\vartheta(y)\cdot \al_j>0$.
Therefore $\vartheta(y)\in W_\af^-$.

For all $\al\in R_P^+$ we have
\begin{align*}
  \vartheta(y)\cdot \al &= (\tau_j(y) \al^*)^* =(\tau_j(y
  \tau_{j^*}(\al^*)))^*.
\end{align*}
We have $\tau_{j^*}(\al^*)\in R_P^+$, so that
$\beta=y\cdot\tau_{j^*}(\al^*)\in R^+$. Since $j\in I$ is
cominuscule, $\al_j$ has multiplicity at most one in $\beta$.
Therefore $\tau_j(\beta)^*\in R_\af^+$, in which $\al_0$ occurs with
multiplicity at most one. It follows that $\tau_j(\beta)^*$ has the
form $\gamma$ or $\delta-\gamma$ for some $\gamma\in R^+$. Therefore
$\vartheta(y)\cdot\al\in R_\af^+$ and
$\vartheta(y)\cdot(\delta-\al)\in R_\af^+$, proving that
$\vartheta(y)\in (W^P)_\af$.

We have $w_0 r_0 w_0 = w_0 r_\theta t_{-\theta^\vee} w_0 = r_\theta
t_{\theta^\vee} = r_0 t_{2\theta^\vee}$. Therefore for every $x\in
W_\af$, there is a $\mu\in Q^\vee$ such that $w_0 x w_0 = x^*
t_\mu$. Using \eqref{E:special} and $w_0^P = (w_0^P)^{-1}$ we have
$$
w_0 \tau_j(y) w_0 = w_0 \tau_j y \tau_j^{-1} w_0 =
w_0^Pv^{-1}_j\tau_jy \tau_j^{-1}v_jw_0^P = w_0^P t_{-\omega_j^\vee}
y t_{\omega_j^\vee} w_0^P = w_0^P y w_0^P t_\mu
$$
for some $\mu \in Q^\vee$.  Thus $\vartheta(y)=  w_0^P y
w_0^Pt_\lambda$ for some $\lambda \in Q^\vee$.  But clearly we have
$\pi_P(w_0^P y w_0^P) = \pi_P(w_0^P y)$, giving us the last
statement of the proposition.

The map $y \mapsto w_0^P y$ induces an involution on $W^P$. Since
$\sigma_P^y$ is a $S[q,q^{-1}]$-basis of $QH^T(G/P)[q^{-1}]$ we
conclude by Theorem~\ref{thm:Pmain} that $\xi_{\vartheta(y)}$ is a
$S[\xi_{\pi_P(t_\la)}^{\pm 1}\mid \lambda \in \tQ]$-basis of
$(H_T(\Gr_G)/J_P)[\xi_{\pi_P(t_\la)}^{-1}\mid \la\in \tQ]$.
\end{proof}

\begin{remark}
Since it is defined using automorphisms of the affine Dynkin
diagram, the map $\vartheta$ induces an isomorphism of the Bruhat
order on $W^P$ with that on its image.
\end{remark}

\begin{ex} \label{X:strange}
Let $G=SL(7)$, $j=4$, and $y=r_4 r_5 r_2 r_3 r_4\in W^P$, which in
one-line notation (that is, the list $y(1),y(2),\dotsc,y(7)$,
viewing $y$ as a permutation of $\{1,2,\dotsc,7\}$) is $y=(1356\mid
247)$ and therefore corresponds to the partition
$(6,5,3,1)-(4,3,2,1)=(2,2,1,0)$ inside the $4\times 3$ rectangle.
The above reduced decomposition of $y$ is obtained by the columnwise
reading of simple reflections in the following picture of the French
diagram of the $(2,2,1,0)$, where the cell $(x_1,x_2)$ contains the
value $j+x_1-x_2$ where the lower left cell is indexed $(1,1)$ and
the cells are indexed by integer lattice points in the first
quadrant of the Cartesian plane. In general if $\mu$ is the
partition denote the corresponding element of $W^P$ by $w_\mu$.
\begin{align*}
\tableau{2 \\ 3&4 \\ 4& 5}	\qquad w_{(2,2,1,0)} = r_4 r_5 r_2 r_3 r_4.
\end{align*}
We have $\tau_j(y)=r_0r_1r_5r_6r_0$ and
$\vartheta(y)=r_0r_6r_2r_1r_0=wt_\la$ where
$\la=-\om_2^\vee-\om_5^\vee$ and $w=r_\theta r_6r_2r_1r_\theta$,
which in one-line notation is $w=(6724\mid 513)$. Then
$\pi_P(t_\la)=r_2r_3r_1r_2 r_6r_5 t_\la$ and $\pi_P(w)=(2467|135)$
which corresponds to the partition $(7,6,4,2)-(4,3,2,1)=(3,3,2,1)$.
\end{ex}

\subsection{Strange duality}
In~\cite{CMP}, Chaput, Manivel and Perrin study a {\it strange
duality} involution on $QH^*(G/P)[q,q^{-1}]$. The final statement of
Proposition~\ref{P:section} suggests a relationship between strange
duality and Theorem~\ref{thm:Pmain}.

\begin{thm}[{\cite[Theorem 4.1]{CMP}}] \label{thm:CMP}
Let $P \subset G$ be a cominuscule parabolic subgroup with
$I_P=I\setminus\{j\}$ and for $w\in W^P$ let $\delta(w)$ be the
number of times $r_j$ appears in some (and thus any) reduced
decomposition of $w$. Then there exists a function $\zeta:W \to \R$
such that
$$
q \mapsto q^{-1} \ \ \ \ \ \sigma_P^{w} \mapsto
\zeta(w)q^{-\delta(w)} \sigma_P^{\pi_P(w_0^P w)}
$$
defines an involutive ring automorphism of $QH^*(G/P)[q^{-1}]
\otimes_\Z \R$.
\end{thm}

In general the values $\zeta(w)$ can be irrational algebraic
numbers, but for $G = SL(n)$, $\zeta(w) = 1$ for all $w \in W$.

One may check that Example \ref{X:strange} agrees with the explicit
description in \cite{CMP} of strange duality on the Grassmannian in
terms of partitions and their Durfee square.

\subsection{The homomorphism of Lapointe and Morse}
Suppose now that $G/P$ is the Grassmannian $\Gr(j,\C^n)=SL_n/P$.
Lapointe and Morse defined a map which, after various
identifications, can be interpreted as a surjective ring
homomorphism $H_*(\Gr_{SL_n})\to QH^*(\Gr(j,\C^n))$. We shall
explain their map in terms of strange duality and the parabolic
Peterson Theorem (Theorem \ref{thm:Pmain}).

For this section let $k=n-1$. In \cite{LLM}, motivated by Macdonald
theory, Lapointe, Lascoux, and Morse defined a family of symmetric
functions $s_\la^{(k)}$ called $k$-Schur functions. They form a
basis of the ring $\Z[h_1,\dotsc,h_k]$, where $h_i$ is the
homogeneous symmetric function. The $k$-Schur basis is indexed by
$k$-bounded partitions, that is, partitions $\la$ such that $\la_1
\le k$.

The homomorphism of Lapointe and Morse may be described as follows.

\begin{thm}\label{thm:LM} \cite{LM:QH} There is a surjective ring
homomorphism $\Z[h_1,\dotsc,h_{n-1}]\to QH^*(\Gr(j,\C^n))$ such that
for any $(n-1)$-bounded partition $\la$, the $(n-1)$-Schur function
$s_\la^{(n-1)}$ maps to $0$ or a power of $q$ times a single quantum
Schubert class. Moreover,
\begin{enumerate}
\item If $\la$ fits inside the
$(n-j)\times j$ rectangle then $s_\la^{(n-1)}\mapsto
\sigma_P^{w_{\la^t}}$ where $\la^t$ is the transpose of the
partition $\la$.
\item If $\la_1 >j$ then $s_\la^{(n-1)}\mapsto 0$.
\end{enumerate}
\end{thm}

The above rules specify the map except when $\la$ consists of some
number of parts of size $j$ followed by a partition contained in the
$(n-j)\times j$ rectangle; in that case one must use a straightening
process to determine the image Schubert class explicitly; see
\cite{LM:QH}.

Bott \cite{Bott} gave an explicit realization of $H_*(\Gr_{SL_n})$
by the ring $\Z[h_1,\dotsc,h_{n-1}]$. Lam \cite{Lam} proved that the
$(n-1)$-Schur functions are the Schubert basis of $H_*(\Gr_{SL_n})$.
To make the identification explicit, we recall a bijection
\cite[Proposition 47]{LM:aff} denoted here by $\la\mapsto
w_\la^\af$, from $(n-1)$-bounded partitions $W_\af^-$, where $W_\af$
is the affine Weyl group for $G=SL_n$. See \cite{LLMS} for
alternative descriptions of this bijection.

Given the $(n-1)$-bounded partition $\la$, we place the value
$x_1-x_2 \mod n$ into the cell $(x_1,x_2)$ in the diagram of $\la$
in a manner similar to the definition of $w_\la$ in Example
\ref{X:strange}.

These entries are then used as indices for simple reflections in a
reduced decomposition of an element $w_\la^\af\in W_\af^-$, reading
the rows in order from the top row to the bottom row, reading within
each row from right to left.

\begin{ex} Let $n=7$ and $\la=(3,2)$. Then the filled diagram of
$\la$ is given by
\begin{align*}
  \tableau{6&0\\0&1&2}
\end{align*}
and $w_\la^\af=r_0 r_6 r_2 r_1 r_0$.
\end{ex}

\begin{thm} \label{thm:Lam} \cite{Lam} Under Bott's isomorphism
$H_*(\Gr_{SL_n})\cong \Z[h_1,\dotsc,h_{n-1}]$, the Schubert class
$\xi_{w_\la^\af}$ maps to the $(n-1)$-Schur function $s_\la^{(n-1)}$
for every $(n-1)$-bounded partition $\la$.
\end{thm}

Combining Theorem \ref{thm:LM} specialized at $q=1$ and Theorem
\ref{thm:Lam} one obtains the Lapointe-Morse ring homomorphism
$\Psi_{LM}:H_*(\Gr_{SL_n})\to QH^*(\Gr(j,\C^n))|_{q=1}$.

On the other hand, combining strange duality and the parabolic
Peterson Theorem we have the following result.

\begin{prop}
Let $G = SL_n$ and $P \subset G$ be a maximal parabolic subgroup
with $I_P=I\setminus\{j\}$. Then there is a surjective ring
homomorphism $\Psi:H_*(\Gr_G)\to QH^*(G/P)|_{q = 1}$ defined by
$$
\xi_{x} \mapsto \begin{cases} \sigma_P^y & \mbox{if $x = \vartheta
(y)\pi_P(t_\lambda)$ for some $y \in W^P$ and $\lambda \in Q^\vee$,} \\
0 & \mbox{otherwise.}
\end{cases}
$$
Moreover $\Psi=\Psi_{LM}$.
\end{prop}
\begin{proof}
$\Psi$ is the composition of the nonequivariant specialization of
the map of Theorem~\ref{thm:Pmain} and the map of Theorem
\ref{thm:CMP} specialized at $q=1$. Together with Proposition
\ref{P:section} it follows that $\Psi$ is a surjective ring
homomorphism.

To prove $\Psi=\Psi_{LM}$ it suffices to check agreement on algebra
generators. For $0\le m\le n-1$ let $h_{[m]} = r_{m-1}\dotsm
r_2r_1r_0\in W_\af^-$; $\xi_{h_{[m]}}$ is the Bott generator
corresponding to the symmetric function $h_m$ and to the
$(n-1)$-bounded partition having a single row of size $m$.

Let $y\in W^P$. Let $\la$ be the partition contained in the
$(n-j)\times j$ rectangle such that $y = w_{\la^t}$. It is easy to
check from the definitions that $\vartheta(y)=w_\la^\af$.
Consequently $\Psi$ and $\Psi_{LM}$ agree on $\xi_{\vartheta(y)}$
for $y\in W^P$. Since $W^P$ contains the elements
$c_{[m]}=r_{j-m+1}\dotsm r_{j-1}r_j$ for $0\le m\le j$ and
$\vartheta(c_{[m]})=h_{[m]}$ for $0\le m\le j$, $\Psi=\Psi_{LM}$ on
the generators $\xi_{h_{[m]}}$ for $0\le m\le j$. Finally, both
$\Psi$ and $\Psi_{LM}$ send $\xi_{h_{[m]}}$ to zero for $j+1\le m\le
n-1$.
\end{proof}

\begin{ex} Let $n=7$, $j=4$ and choose $\la=(3,2,0)$ in the
$3\times 4$ rectangle. Then $\la^t=(2,2,1,0)$ fits in the $4\times
3$ rectangle and $w_{\la^t}\in W^P$ is given by the element $y$ of
Example \ref{X:strange}. The element $\vartheta(y)$ is given by
$w_\la^\af$, which appears in the two previous examples.
\end{ex}

\begin{remark}
The ``Pieri formula'' for $H_*(\Gr_{SL_n})$ was given in
\cite{LLMS}, and agrees with the $k$-Pieri rule of Lapointe and
Morse \cite{LM:aff}.  The image of this Pieri rule under $\Psi$ is
exactly the quantum Pieri rule of $QH^*(\Gr(j,\C^n))$; see
\cite{Ber}.
\end{remark}

\end{document}